\documentclass[a4paper, 11pt, headings = small, abstract = true]{scrartcl}
\usepackage[sc]{mathpazo}    
\usepackage[scaled]{helvet}  
\usepackage{eulervm}		
\usepackage[utf8]{inputenc}
\usepackage[T1]{fontenc}
\usepackage[english]{babel}
\usepackage{enumerate}
\usepackage{amssymb,amsfonts,amsmath,amsthm}
\usepackage{mathabx} 
\usepackage{mathtools}
\usepackage{color}

\usepackage[noblocks]{authblk}

\usepackage[linktocpage, colorlinks]{hyperref} 
\usepackage{color}
\definecolor{darkred}{rgb}{0.5,0,0}
\definecolor{darkgreen}{rgb}{0,0.5,0}
\definecolor{darkblue}{rgb}{0,0,0.5} 	\hypersetup{colorlinks,linkcolor=darkblue,filecolor=darkgreen,urlcolor=darkred,citecolor=darkblue}
\hypersetup{
	pdftitle={Observability inequalities for infinite-dimensional systems in Banach spaces and unique determination of a singular potential from boundary data},
	pdfauthor={Clemens Bombach},
}

\numberwithin{equation}{section}
\mathtoolsset{showonlyrefs}

\DeclareMathOperator{\real}{Re}
\DeclareMathOperator{\imag}{Im}
\DeclareMathOperator{\spt}{supp}
\DeclareMathOperator{\dom}{dom}
\DeclareMathOperator*{\esssup}{ess\,sup}
\DeclareMathOperator{\sgn}{sgn}

\let\originali\i

\renewcommand{\i}{\ifthenelse{\boolean{mmode}}{\mathbf i}{\originali}}
\renewcommand{\a}{\mathfrak{a}}
\renewcommand{\b}{\mathfrak{b}}

\newcommand{\q}{\mathfrak{q}}
\renewcommand{\theta}{\vartheta}
\renewcommand{\phi}{\varphi}
\renewcommand{\d}{\mathfrak{d}}

\renewcommand{\L}{\mathcal{L}}
\newcommand{\ii}{\mathrm{i}}
\newcommand{\e}{\mathbf{e}}
\newcommand{\euler}{\mathrm{e}}

\newcommand{\T}{\mathbb{T}}
\newcommand{\D}{\mathcal{D}}

\newcommand{\F}{\mathcal{F}}

\newcommand{\C}{\mathbb{C}}
\newcommand{\R}{\mathbb{R}}
\newcommand{\N}{\mathbb{N}}

\newcommand{\K}{\mathcal{K}}
\newcommand{\diff}{\mathrm{d}}

\newcommand{\trans}{\top}

\newcommand{\I}{\mathfrak{I}}
\renewcommand{\S}{\mathcal{S}}
\newcommand{\abs}[1]{{\left\lvert#1\right\rvert}}
\newcommand{\biggabs}[1]{{\bigg\lvert#1\bigg\rvert}}
\newcommand{\norm}[1]{{\left\lVert#1\right\rVert}}
\newcommand{\dual}[1]{\langle#1\rangle}

\newcommand{\loc}{\mathrm{loc}}

\newtheorem{theorem}{Theorem}[section]
\newtheorem{lemma}[theorem]{Lemma}
\newtheorem{proposition}[theorem]{Proposition}

\theoremstyle{definition}
\newtheorem{definition}[theorem]{Definition}
\newtheorem{example}[theorem]{Example}

\theoremstyle{remark}
\newtheorem{remark}[theorem]{Remark}

\title{Unique determination of a Kato class potential from boundary data}
\author{Clemens Bombach}
\affil{Technische Universit\"at Chemnitz, Fakult\"at f\"ur Mathematik, Reichenhainer Str. 39, 09126 Chemnitz, Germany}

\begin{document}

\maketitle
\begin{abstract}
We prove that a Kato class potential $V$ defined on an open, bounded set in $\R^3$ with Lipschitz boundary is uniquely determined by the Dirichlet-to-Neumann operator associated to the equation	$-\Delta u + Vu = 0\,.$
\end{abstract}

\section{Introduction}

\subsection*{Statement of the main result}

Let $n\geq 3$ and $U$ be an open, bounded set in $\R^n$ with Lipschitz boundary. We say that a function $V \in L^1_\loc(\R^n)$ lies in the Kato class $\K_n$ if and only if
\begin{equation*}
	\lim_{r \to 0} \sup_{x \in \R^n}\, \int\limits_{\abs{x - y} < r} \abs{x-y}^{2 -n}\abs{V(y)}\,\diff y = 0\,.
\end{equation*}
Given a function $V$ defined on some open, bounded set $U \subseteq \R^n$ with Lipschitz boundary, we consider the equation
\begin{equation}\label{eq:schroedinger}
	-\Delta u + Vu = 0~ \text{in}~U\, .
\end{equation}
If $V$ is in a class of sufficiently regular functions, one can find weak solutions $u \in H^1(U)$ of this PDE. Setting $\Gamma = \partial U$, and assuming that $0 \notin \sigma(A_V)$, where $A_V$ is the realization of $-\Delta + V$ with Dirichlet boundary conditions on $L^2(U)$, we have a well-defined solution operator $P_V:H^{1/2}(\Gamma) \to H^1(U)$ mapping $\phi$ to the solution $u$ of the Dirichlet problem
\begin{equation}\label{eq:Dirichlet_problem}
	\begin{aligned}
		-\Delta u + Vu &= 0~ \text{in}~ U \\
		u|_\Gamma &= \phi\,.
	\end{aligned}
\end{equation}
Using this operator $P_V$, we define the \emph{Dirichlet-to-Neumann operator} as the operator
\[
\Lambda_V:H^{1/2}(\Gamma) \to H^{-1/2}(\Gamma),\quad \phi \mapsto \partial_\nu P_V\phi\,,
\]
where $\partial_\nu$ denotes the weak derivate on $\Gamma$ in the direction of the outer normal vector field. For a given class of potentials $\mathcal{F}$, a natural question is whether the Dirichlet-to-Neumann operator associated to a potential determines the potential uniquely. That is, we ask whether the mapping \[ \mathcal F \ni V \mapsto \Lambda_V \] is injective. In this work, we answer this question in the affirmative for the Kato class $\K_3$.

In the work \cite{chanillo1990problem} the authors proved that $\Lambda_V$ uniquely determines $V$ if $V$ is either in $L^p(U)$ for $p > n/2$ or has small norm in certain Morrey spaces (also called Fefferman-Phong classes). An extension of this result to potentials in $L^{n/2}(U)$ was obtained by Nachman. His (previously unpublished) proof can be found in the paper \cite{dos2013determining} by Dos Santos Ferreira, Kenig and Salo, where the authors considered Dirichlet-to-Neumann operators defined on a certain class of admissible manifolds.

Our main result is as follows:
\begin{theorem}\label{main_theorem}
	Let $U \subseteq \R^3$ be an open, bounded set with Lipschitz boundary and $V_1, V_2$ be in $\K_3$ with $\spt V_1,V_2 \subseteq  U$. Suppose that $0 \notin \sigma(A_{V_j})$ where $A_{V_j}$ is the realization of the Schrödinger operator $-\Delta + V$ on $L^2(U)$ with Dirichlet boundary conditions $j=1,2$. If ${\Lambda}_{V_1} = {\Lambda}_{V_2}$, then it follows that $V_1 = V_2$.
\end{theorem}

We remark that neither the inclusion $L^{n/2} \subseteq \K_n$ nor the inclusion $\K_n \subseteq L^{n/2}$ are true. Thus, our result complements the aforementioned earlier results.

\subsection*{Related problems}

The inverse problem considered here is closely related to the Calder\'{o}n problem for isotropic conductivities, where instead of the Schrödinger operator $-\Delta + V$ an elliptic operator $-{\nabla \cdot \sigma \nabla}$ is studied and the goal is to determine the scalar-valued function $\sigma$ from the associated Dirichlet-to-Neumann operator $\Lambda_\sigma$. This function $\sigma$ can be viewed as an electrical conductivity and $\Lambda_\sigma$ corresponds to voltage and current measurements at the boundary.

In the paper \cite{sylvester1987global} by Sylvester and Uhlmann, uniqueness was proven for 
\[ \sigma \in C^\infty(\overline U) \quad\text{and} \quad n \geq 3 \,.\]
An important step in their proof is to transform the equation $-\nabla \cdot \sigma \nabla u = 0$ into the equivalent Schrödinger-type equation $-\Delta u + Vu = 0$ with $V = \sigma^{-1/2}\Delta(\sigma^{1/2})$. Sylvester and Uhlmann then prove that $V$ is uniquely determined by $\Lambda_V$. To achieve this, they construct rapidly oscillating solutions to the equation \eqref{eq:schroedinger} that behave like $\euler^{\ii z\cdot x}$ for certain large complex vectors $z$. These solutions are called Complex Geometrical Optics solutions or CGO solutions.

In fact, for the Sylvester-Uhlmann construction of CGO solutions, it is only \linebreak required that $V \in L^\infty(U)$. Using a result of Kohn and Vogelius \cite{kohn1984determining} which states that $\Lambda_\sigma$ uniquely determines $\sigma|_\Gamma$ and all its derivatives restricted to $\Gamma$, they then showed that $\Lambda_\sigma$ also uniquely determines $\sigma$ in the interior of $U$.

Since then, there have been various attempts to extend this result. The two-dimen-sional case was solved for $\sigma \in L^\infty$ in \cite{astala2006calderon}. For $n \geq 3$, extensive work has been done for various classes of nonsmooth conductivities. The following is a partial overview of results in that direction: In the work \cite{brown1996global}, the smoothness assumption on $\sigma$ was reduced to $\nabla \sigma \in C^{3/2 + \varepsilon}$ for arbitrary $\varepsilon > 0$. Haberman and Tataru (see \cite{haberman2013uniqueness}) proved uniqueness for small conductivities $\sigma \in W^{1,\infty}$. The smallness assumption was removed by Caro and Rogers in \cite{caro2016global}. Also, a generalization to $\sigma \in W^{1,n}$ for $n = 3,4$ was obtained by Haberman in \cite{haberman2015uniqueness} who also proved uniqueness in dimension $5,6$ under only slightly stronger Sobolev-type conditions. The range of admissible Sobolev exponents was further extended in the works \cite{ham2019uniqueness} and \cite{ponce2020bilinear}. Similarily, the regularity assumptions needed for the boundary determination result originally established in \cite{kohn1984determining} were greatly reduced by Brown in \cite{brown2001recovering}.

As shown in \cite{ham2019uniqueness}, a consequence of Haberman's work \cite{haberman2015uniqueness} is that for $V_1,V_2 \in W^{-1,3}$ the identity $\Lambda_{V_1} = \Lambda_{V_2}$ implies that $V_1 = V_2$. Our result also holds for some $V_1,V_2 \notin W^{-1,3}$, see Example \ref{example}.

The inverse problem considered here appears to be closely related to the unique continuation problem since both involve Carleman estimates, which are uniform $L^p$-estimates for the conjugated Laplacian $\euler^{-\ii \phi}\Delta \euler^{\ii \phi}$ where $\phi$ is taken from a certain set of complex phase functions. In this work, we will only consider complex-valued linear phase functions. Regarding singular potentials, we mention the result of Sawyer \cite{sawyer1984unique} who proved the unique continuation property for the differential inequality
\[
\abs{\Delta u} \leq \abs{Vu}
\] 
with $V$ in the Kato class $\K_d$ for $d \geq 3$. "Unique continuation property" in this context means that functions $u \in W^{2,1}(U)$ that satisfy the above differential inequality and vanish on some open subset of $U$ must in fact vanish everywhere on $U$. It was conjectured in \cite{simon1982schrodinger} that unique continuation holds for the Kato class in all dimensions. This conjecture is still not resolved so far, with only a positive result in the case of radial Kato potentials obtained in \cite{fabes1990partial}. Unique continuation for $V \in L^{n/2}$ was proven in \cite{kenig1987uniform}. For results regarding general second-order elliptic operators with non-smooth coefficients we refer to \cite{koch2001carleman}.

Another related problem is to prove that for a periodic potential $V \in \K_n$, the Schrödinger operator $-\Delta + V$ on $L^2(\R^n)$ has absolutely continuous spectrum. For $n=2$ and $n=3$, this was carried out by Shen in \cite{shen2001absolutekato}. We will use some techniques similar to the ones used in the latter work and the article \cite{shen2001absolute}, where Shen treated the case $V \in L^{n/2}(\T^n)$ for $n \geq 3$.

\subsection*{Construction of CGO solutions}

The main technical obstruction in the proof of Theorem \ref{main_theorem} is the construction of CGO solutions. Therefore, most of the work is concerned with proving the following existence theorem:
\begin{theorem}\label{CGO_theorem}
	Let $V \in \K_3$ with $\spt V \subseteq U$, where $U \subseteq \R^3$ is open and bounded with Lipschitz boundary. There exists a $s_0 > 0$ such that for all $z \in \C^3$ with $\abs{z} > s_0$ and $z \cdot z = 0$ we can find elements $u_z, r_z \in H^1(U)$ such that
	\begin{equation}\label{CGO_schroedinger}
		-\Delta u_z + Vu_z = 0 
	\end{equation}
	and
	\begin{equation}\label{CGO_form}
		u_z(x) = \euler^{\ii z \cdot x}(1 + r_z(x))
	\end{equation}
	for almost every $x \in U$. The error term $r_z$ can be chosen in such a way that
	\begin{equation*}
		\quad \lim_{z \to \infty} \norm{r_z}_{L^2(U)} = 0\,.
	\end{equation*}
\end{theorem}
Once Theorem \ref{CGO_theorem} has been proven, Theorem \ref{main_theorem} follows by a well-known argument.

The general strategy is as follows: A short calculation shows that the error term $r_z$ can be found by finding a $L^2-$small solution of the equation
\begin{equation}\label{eq:CGO_equation}
	-\euler^{-\ii z\cdot x}\Delta (\euler^{\ii z\cdot x}r_z(x)) + V(x)r_z(x) = - V(x)\,,
\end{equation}
where $z \in \C^3$ with $\abs{z} > s_0 \in \R$ such that $z \cdot z = 0$. As in the works \cite{sylvester1987global}, \cite{haberman2013uniqueness} and \cite{haberman2015uniqueness}, we obtain $r_z$ by a pertubation argument. For this purpose, it is useful to construct a suitable inverse of the conjugated Laplacian $-\euler^{-\ii z\cdot x}\Delta \euler^{\ii z\cdot x}$. This is a constant coefficient differential operator with symbol $\mathfrak{p}_z(\xi) = \abs{\xi}^2 + 2z\cdot \xi$. We may construct a right inverse $G_z$ to $\mathfrak{p}_z(D)$ by convolving with $\mathfrak{E}_z = \F^{-1}(\mathfrak{p}_z)^{-1}$ where $\F^{-1}$ denotes the inverse Fourier transform.

The key identity we need to prove is that for compactly supported $V,W \in \K_3$, we have
\begin{equation}\label{eq:intro_estimate}
	\lim_{\tau \to \infty} \| {\abs{V}^{1/2}}G_z\abs{W}^{1/2} \|_{L^2(U) \to L^2(U)} =0\,,
\end{equation}
as this will be used to show that the error term $r_z$ goes to $0$ as $z \to \infty$.

To this end, we observe that by a change of coordinates, it can be deduced from an estimate proven in \cite{shen2001absolutekato} there exists an absolute constant $A>0$ and a $s_0 > 0$ such that for all $z \in \C^3$ with $z\cdot z = 0$ we have
\begin{equation}\label{eq:pointwise_bound}
	\abs{\mathfrak{E}_z(x)} \leq A\abs{x}^{-1} \quad (x \in \R^3)\,.
\end{equation}

Equipped with this pointwise bound, we obtain that
\begin{equation*}
	\norm{\abs{V}G_z}_{L^1(\R^3) \to L^1(\R^3)} \leq C(V), \quad \norm{G_z\abs{V}}_{L^\infty(\R^3) \to L^\infty(\R^3)} \leq C(V),
\end{equation*}
with a constant depending only on $V$. This allows us to deduce \eqref{eq:intro_estimate} by applying the Stein interpolation theorem, an estimate from \cite{sylvester1987global} and the fact that compactly supported Kato class potentials can be approximated by test functions in a suitable norm (Lemma \ref{lem:approx}).

\subsection*{Summary of contents and remarks on notation}

In the preliminary section \ref{sec:kato} we define the Kato class $\K_n$ and give a proof of the aforementioned approximation property Lemma \ref{lem:approx}. In the following Section \ref{sec:CGO} we prove Theorem \ref{CGO_theorem}, following the strategy that we outlined in the preceding subsection. In  Section \ref{sec:DTN}, we construct the Dirichlet-to-Neumann operator associated to a potential $V \in \K_n$ via form methods. After this, we are in an excellent position to prove the main result, Theorem \ref{main_theorem}. The proof is carried out in Section \ref{sec:main}.

Moreover, this work contains three appendices. In Appendix \ref{sec:bessel} we recall some useful properties of the Riesz and Bessel potentials. In Appendix \ref{sec:point}, we give a proof of the pointwise estimate used in Section \ref{sec:CGO} with an explicit constant $A$. Appendix \ref{sec:extensions} discusses possible generalisations to higher dimensions.

Throughout this work, we adopt the convention that $C$ is an \emph{absolute constant} that is only allowed to depend on certain fixed parameters depending on the context. The exact value of $C$ may change from line to line. If we want to emphasize that $C$ does depend on a certain other quantity $x$, we write $C = C_x$. Conversely, if we want to emphasize that $C$ is independent of $x$, we write $C \neq C_x$.

By $n$ and $d$ we denote fixed natural numbers such that $n \geq 3$ and $d \geq 2$, denoting the dimension of the underlying Euclidean space. Although the main theorems are only stated and proven for $n = 3$, we have formulated preparatory lemmas and propositions in general dimension whenever feasible.

\section{The Kato class}\label{sec:kato}

Suppose that $d \geq 2$ and $0 < \lambda \leq d$. For $x \in \R^d$, we set
\begin{equation*}
	k_\lambda(x) = \begin{cases}
		1/\abs{x}^{d- \lambda}: \quad d \geq 3,  \\
		\abs{\log(1/\abs{x})}: \quad \lambda = d = 2
	\end{cases}\,.
\end{equation*}
Let $V \in L^1_\loc(\R^d)$. We say that $V$ is an element of the Kato class $\K_d$ if and only if
\begin{equation*}
	\lim_{r \to 0} \sup_{x \in \R^d} \int\limits_{\abs{x - y} < r} k_{2}(x-y)\abs{V(y)}\,\diff y = 0\,.
\end{equation*}
\begin{remark}\label{rem:ss}
	In this remark we discuss some generalizations of the Kato conditon.
	\begin{enumerate}[a)]
		\item One can view the class $\K_d$ as a special case of the Schechter-Stummel classes $M_{\alpha,p}$, originally definied in \cite{schechter1971spectra}, see also \cite[p. 453]{simon1982schrodinger}. For $0 < \alpha \leq d$, $1 \leq p < \infty,$ we have that $V \in M_{\alpha,p}$ if and only if
		\begin{equation*}
			\lim_{r \to 0} \sup_{x \in \R^d} \int\limits_{\abs{x - y} < r} k_{\alpha}(x-y)\abs{V(y)}^p\,\diff y = 0\,.
		\end{equation*}
		From Hölder's inequality, the inclusion
		\begin{equation*}
			M_{\alpha,p} \subseteq M_{\beta,q}
		\end{equation*}
		follows if $\alpha/p < \beta/q$ and $ d \geq 3$ (as remarked in \cite{kuwae2022p} this inequality is displayed with a sign error in \cite{simon1982schrodinger}). We note that $\K_d = M_{2,1}$.
		\item In \cite{zheng2009higher}, a definition of fractional order Kato classes $K_\alpha = M_{2\alpha,1}$ for $\alpha > 0$ was given and the characterisation of these potentials was applied to study higher order Schrödinger operators. 
		\item The Kato class $\K_d$ may be characterized by the fact that $V \in \K_d$ if and only if
		\begin{equation*}
			\lim_{t \to 0} \sup_{x \in \R^d} \int\limits_{\R^d} \bigg( \int\limits_0^t p_s(x,y)\,\diff s \bigg)\abs{V(y)}\,\diff y = 0\,,
		\end{equation*}
		with $p_t(x,y) = (\F^{-1}e^{-t\abs{\cdot}^2})(x-y)$ denoting the heat kernel.
		In \cite{kuwae2022p} the $L^p$-Kato class $\K^p_d$ is considered which consists of Borel measures on $\R^d$ characterised by
		\begin{equation*}
			\mu \in \K^p_d \iff \lim_{r \to 0} \sup_{x \in \R^d} \int_{\abs{x - y} < r} k_{2}(x-y)^p\,\mu(\diff y) = 0\,.
		\end{equation*}
		The authors of \cite{kuwae2022p} show that if $d - p(d-2) > 0$, then this condition is equivalent to
		\begin{equation*}
			\lim_{t \to 0} \sup_{x \in \R^d} \int\limits_{\R^d} \bigg( \int\limits_0^t p_s(x,y)\,\diff s \bigg)^p\,\mu(\diff y) = 0\,.
		\end{equation*}
		\item In \cite{stollmann1996perturbation}, the authors generalize the Kato condition to the setting of measure pertubations $\mu$ of abstract Dirichlet forms $\mathfrak{h}$. One of the principal results established in their work (\cite[Lemma 2.2]{stollmann1996perturbation}) is that the operator sequence $(m(H - m)^{-1}\mu)_{m \in \N}$ vaguely converges to $\mu$ where $H$ is the operator generated by the Dirichlet form $\mathfrak{h}$. The classical case is the one where $H = -\Delta$ and $\mu$ is absolutely continuous. In this special situation, Lemma \ref{lem:approx} gives a somewhat stronger approximation property.
	\end{enumerate}
\end{remark}

We denote by $\K_d^c$   the set of compactly supported elements of $\K_d$. For $V \in \K_d^c$, the global Kato norm
\begin{equation*}
	\norm{V}_{\K_d} = \sup_{x \in \R^d} \int\limits_{\R^d}k_{2}(x-y)\abs{V(y)}\,\diff y
\end{equation*}
is finite.

Let $n \geq 3$. Elements of $\K_n^c$ can be approximated by elements of the test function space $\D(\R^n)$ in this norm,  see \cite[Theorem 3.3]{zheng2009higher}. By a careful inspection of the arguments given in \cite[Theorem 3.1,Theorem 3.3]{zheng2009higher}, it is possible to deduce that we may choose the approximating test functions such that their global Kato norm is bounded by $\norm{V}_{\K_n}$. More precisely, we obtain:
\begin{lemma}\label{lem:approx}
	Suppose that $V \in \K_n^c$. Then, for each $\varepsilon  > 0$ there exists a $V_\flat\in \D(\R^n)$ such that
	\begin{equation}
		\norm{V_\flat}_{\K_n} \leq \norm{V}_{\K_n}, \quad \norm{V - V_\flat}_{\K_n} < \varepsilon\,.
	\end{equation}
\end{lemma}
\begin{proof}
	For the reader's convenience, we give a self-contained proof. Since $\spt V$ is compact, it follows that $V \in L^1(\R^n)$.
	Let $\phi \in \D(\R^n)$ such that $\phi \geq 0$ and $\norm{\phi}_{L^1} = 1$. For each $\delta > 0$ we define $\phi_\delta = \delta^{-n}\phi(\delta^{-1}\cdot)$ and $V_\delta = V * \phi_\delta$. Clearly, it holds that $V_\delta \in \D(\R^n)$. Since $(\phi_\delta)_{(\delta > 0)}$ is an approximate unit in the convolution algebra $L^1(\R^n)$, we have $V_\delta \to V$ in $L^1(\R^n)$ as $\delta \to 0$.
	
	By Fubini's theorem, we have for each $x \in \R^n$ that
	\begin{align*}
		\int\limits_{\R^n} \frac{\abs{V_\delta(y)}\,\diff y}{\abs{x-y}^{n - 2}} &= \,\int\limits_{\R^n}  \bigg(\,\int\limits_{\R^n} V(y-z)\phi_\delta(z)\,\diff z \bigg) \frac{\,\diff y}{\abs{x -y}^{n-2}} \\ &= \int \limits_{\R^n} \bigg(\int\limits_{\R^n}\abs{V(y-z)}\frac{\,\diff y}{\abs{x -y}^{n-2}}\bigg) \phi_\delta(z)\,\diff z\,.
	\end{align*}
	
	Taking the supremum over all $x \in \R^d$, we obtain
	\begin{equation*}
		\norm{V_\delta}_{\K_n} \leq \int \limits_{\R^n} \norm{V(\cdot - z)}_{\K_n} \phi_\delta(z) \,\diff z = \norm{V}_{\K_n}\norm{\phi_\delta}_{L^1(\R^n)} = \norm{V}_{\K_n}\,.
	\end{equation*}
	By another application of Fubini's theorem, it follows that
	\begin{align*}
		\sup_{x \in \R^n}  \int\limits_{\abs{x -y} < r}  \frac{\abs{V_\delta(y)}\,\diff y}{\abs{x -y}^{n - 2}} &\leq \int\limits_{\R^n} \sup_{x \in \R^n}  \int\limits_{\abs{x -y} < r}  \frac{\abs{V(y - z)}\,\diff y}{\abs{x -y}^{n - 2}}  \phi_\delta(z)\,\diff z \\ &= \sup_{x \in \R^n}  \int\limits_{\abs{x -y} < r}  \frac{\abs{V(y)}\,\diff y}{\abs{x -y}^{n - 2}} 
	\end{align*}
	where we made the substitution $y -z \mapsto y$ in the second line and used the observation that
	\begin{equation*}
		\sup_{x \in \R^n}  \int\limits_{\abs{x -y - z} < r}  \frac{\abs{V(y)}\,\diff y}{\abs{x -y - z}^{n - 2}} = \sup_{x \in \R^n}\int\limits_{\abs{x -y} < r}  \frac{\abs{V(y)}\,\diff y}{\abs{x -y}^{n - 2}} \,.
	\end{equation*}
	It follows that
	\begin{align*}
		\sup_{x \in \R^n} \int\limits_{\R^n}\frac{\abs{(V - V_\delta)(y)}\,\diff y}{\abs{x -y}^{n - 2}} \leq& \sup_{x \in \R^n} \int\limits_{\abs{x -y} < r}  \frac{\abs{V(y)}\,\diff y}{\abs{x -y}^{n - 2}}  \\ &\qquad +  \int\limits_{\abs{x -y} < r}  \frac{\abs{V_\delta(y)}\,\diff y}{\abs{x -y}^{n - 2}} 
		+  \int\limits_{\abs{x -y} \geq r}  \frac{\abs{(V - V_\delta)(y)}\,\diff y}{\abs{x -y}^{n - 2}} \\
		\leq& 2\sup_{x \in \R^n}   \int\limits_{\abs{x -y} < r}  \frac{\abs{V(y)}\,\diff y}{\abs{x -y}^{n - 2}}  + r^{2-n}\norm{V - V_\delta}_{L^1}\,.
	\end{align*}
	Let $\varepsilon > 0$. We may choose $r$ such that
	\begin{equation*}
		\sup_{x \in \R^n} \int\limits_{\abs{x -y} < r}  \frac{\abs{V(y)}\,\diff y}{\abs{x -y}^{n - 2}} < \frac{\varepsilon}{4}
	\end{equation*}
	and $\delta > 0$ such that
	\begin{equation*}
		\norm{V - V_\delta}_{L^1(\R^n)} < \frac{\varepsilon}{2r^{2-n}}\,.
	\end{equation*}
	It follows that
	\begin{equation*}
		\norm{V - V_\delta}_{\K_n} < \varepsilon\,. \qedhere
	\end{equation*}	
\end{proof}

\begin{remark}
	It would be interesting to generalize the main result of this paper to arbitrary, not necessarily absolutely continuous Kato class measures (see Remark \ref{rem:ss}). For this purpose, it would be convenient to generalize Lemma \ref{lem:approx} to not necessarily absolutely continous Kato class measures.
	
	In this setting, we would define
	\begin{equation*}
		\norm{\mu}_{\K_n} = \sup_{x \in \R^d} \int\limits_{\R^d}k_{2}(x-y)\,\diff \abs{\mu}(y)
	\end{equation*}
	with $\abs{\mu}$ denoting the total variation measure. Consider, for example, the surface measure $\sigma_{n-1}$ of the unit sphere in $\R^n$. It is well know that $\sigma_{n-1} \in \K_n$. Denoting by $(\phi_{\delta})_{\delta > 0}$ a smooth, nonnegative approximate unit in $L^1(\R^n)$, we have that
	\begin{equation*}
		\abs{\phi_\delta * \sigma_{n-1} - \sigma_{n-1}} = \phi_\delta * \sigma_{n-1} + \sigma_{n-1}\,.
	\end{equation*}
	This shows that $\phi_\delta * \sigma_{n-1}$ does not converge to $\sigma_{n-1}$ in the total variation norm.
	
	Since an essential ingredient of the proof of Lemma \ref{lem:approx} was that $\phi_\delta * V \to V$ in $L^1(\R^n)$, it follows that the proof of Lemma \ref{lem:approx} does not straightforwardly generalize to the setting of Kato class measures. Indeed, it seems unlikely that a strong approximation property such as the one in Lemma \ref{lem:approx} holds for Kato class measures.
\end{remark}

\begin{example} \label{example}
	We consider a compactly supported, nonnegative potential $V:\R^3 \to \R$ such that $V(x) = \abs{x'}^{-2}\abs{\log(x')}^{-\delta}$ for $\abs{x} \leq 1/2$ where $x' = (x_2,x_3)$ and $\delta > 2$. This is an example of a function such that $V \in \K_3$ but $V \notin W^{-1,p}(\R^3)$ for any $p > 2$. In particular $V \notin{W^{-1,3}(\R^3)}$ which is the critical Sobolev space considered in \cite{haberman2015uniqueness, ham2019uniqueness}. Indeed, from \cite[p.455f., Example A]{simon1982schrodinger} we obtain that if $d < n$, $T:\R^n \to \R^d$ linear and $V \circ T \in \K_d$, then $V \in \K_n$. Thus, it suffices to show that $\tilde V: \R^2 \to \R$ given by $\tilde V(z) = V(0,z)$ is an element of $\K_2$, which follows from the fact that
	\begin{align*}
		\sup_{z \in \R^2} \int\limits_{\abs{z-w} \leq r} \tilde V(z)\abs{\log(1/\abs{z-w})}\,\diff z &\leq \int_{\abs{w} \leq r} \abs{w}^{-2}\abs{\log(1/\abs{w})}^{1-\delta}\,\diff w \\
		&\leq C\int\limits_{0}^r t^{-1}\abs{\log(1/t)}^{1-\delta}\,\diff t \\ &\leq C\abs{\log(1/r)}^{2-\delta} \to 0 \quad (r \to 0)\,.
	\end{align*}
	
	Let $p^\prime$ be such that $1/p + 1/p^\prime = 1$. To see that $V \notin W^{-1,p}(\R^3)$, we interpret $V$ as an element of $\D^\prime(\R^3)$ given by
	\begin{equation*}
		\dual{V,\phi} = \int\limits_{\R^3} V(x)\phi(x)\,\diff x, \quad (\phi \in \D(\R^3))\,,
	\end{equation*}
	and remark that since $W^{-1,p}(\R^3) =( W^{1,p^\prime}_0(\R^3))^\prime$ where \[W^{1,p^\prime}_0(\R^3) = \overline{\D(\R^3)}^{\norm{\cdot}_{W^{1,p^\prime}(\R^3)}}\,, \] the norm on $W^{-1,p}(\R^3)$ is given by
	\begin{equation*}
		\norm{V}_{W^{-1,p}(\R^3)} = \sup_{\phi \in \D(\R^3)} \frac{\abs{\dual{V,\phi}}}{\norm{\phi}_{W^{1,p^\prime}(\R^3)}}\,.
	\end{equation*}
	Thus, we show that for each $N \in \N$ there exists a function $\phi \in \D(\R^3)$  such that 
	\begin{equation}\label{eq:ex}
		\abs{\dual{V,\phi}} \geq N\norm{\phi}_{W^{1,p^\prime}}\,.
	\end{equation}
	To this end, let $\varepsilon > 0$, $\eta \in \D(\R)$ such that $\eta = 1$ on $(-1,1)$ and $\psi \in \D(\R^2)$ such that $\psi = 1$ on $B(1)$. For $x \in \R^3$ we define $\phi_\varepsilon(x) = \eta(x_1)\psi(\varepsilon^{-1} x')$. It is clear that
	\begin{align*}
		\norm{\phi_\varepsilon}_{W^{1,p^\prime}} &= \norm{\phi_\varepsilon}_{L^{p'}(\R^3)} + \norm{\nabla \phi_\varepsilon}_{L^{p^\prime}(\R^3;\C^3)} \\ &\approx \varepsilon^{2/p^\prime}\big(\norm{\phi_1}_{L^{p'}(\R^3)} + \norm{\partial_{x_1}\phi_1}_{L^{p'}(\R^3)}\big) \\ &\quad + \varepsilon^{2/p^\prime - 1}\big(\norm{\partial_{x_2}\phi_1}_{L^{p'}(\R^3)} + \norm{\partial_{x_3}\phi_1}_{L^{p'}(\R^3)}\big)
	\end{align*}
	and thus
	\begin{equation}\label{eq:ex1}
		\norm{\phi_\varepsilon}_{W^{1,p^\prime}} \approx (1+\varepsilon^{-1})\varepsilon^{2/p^\prime}\,.
	\end{equation}
	On the other hand, by applying polar coordinates, we obtain that
	\begin{equation}\label{eq:ex2}
		\abs{\dual{V,\phi_\varepsilon}} \geq C\int_{\abs{x'} \leq \varepsilon} \abs{x'}^{-2}\abs{\log(1/\abs{x'})}^{-\delta} \geq C\abs{\log(1/\varepsilon)}^{1-\delta}\,.
	\end{equation}
	We claim that for any given $N \in \N$, we can find $\varepsilon > 0$ such that
	\begin{equation}\label{eq:log}
		\abs{\log(1/\varepsilon)}^{1-\delta} \geq N (1+\varepsilon^{-1})\varepsilon^{2/p^\prime}\,.
	\end{equation}
	Indeed, setting $\varepsilon = \euler^{-t}$ where $t > 0$, the above inequality is equivalent to
	\begin{equation*}
		t^{1-\delta} \geq N(1+\euler^t)\euler^{-2t/p^\prime}
	\end{equation*}
	which is the same as
	\begin{equation}\label{eq:compare}
		\euler^{2t/p^\prime}t^{1-\delta} \geq N(1+\euler^t).
	\end{equation}
	Since $1/p < 1/2$ implies that $2/{p^\prime} > 1$, it follows that the left hand side of \eqref{eq:compare} grows faster than the right hand side as $t \to \infty$ which implies the claim \eqref{eq:log}. It now follows from \eqref{eq:ex1},  \eqref{eq:ex2} and \eqref{eq:log} that \eqref{eq:ex} holds with $\varphi = \varphi_\varepsilon$.
\end{example}

\section{Construction of CGO solutions}\label{sec:CGO}
In this section, we use the pointwise estimates derived in the previous section to deduce the necessary $L^2$-estimates required for the proof of Theorem \ref{CGO_theorem}.

Let $U \subseteq \R^n$ be bounded and open and $z \in \C^n$ such that $z \cdot z = 0$, the $\cdot$ denoting the real inner product. We define
\[
\e_z:\R^n \to \R, \quad x \mapsto \euler^{\ii z \cdot x}
\]
and set $D = -i\nabla$. Recall that $D\e_z = z\e_z$ and $D \cdot D = -\Delta$.
In order to prove Theorem \ref{CGO_theorem}, we need to construct $u_z, r_z \in H^1(U)$ such that
\[
-\Delta u_z + V u_z = 0
\]
and
\[
u_z = \e_z(1 + r_z)
\]
hold. Inserting the second equation \eqref{CGO_form} into equation \eqref{CGO_schroedinger}, we obtain the equation
\begin{equation*}
	(-\Delta + 2z \cdot D + V)r_z = - V.
\end{equation*}
To see this, we note that
\[
-\Delta (\e_z f) = -(\Delta \e_z)f + 2D\e_z\cdot Df - \e_z \Delta f = \e_z(-\Delta + 2z\cdot D)f\,.
\]

We consider the differential operator \[ \mathfrak{p}_z(D) = -\Delta + 2z\cdot D\,, \]
with symbol given by the polynomial
\[
\mathfrak{p}_z(\xi) = -\abs{\xi}^2 + z\cdot \xi, \quad (\xi \in \R^n)\,.
\]
A right inverse  \[ G_z:\mathcal{E}^\prime(\R^n) \to \mathcal{S}'(\R^n) \] to $\mathfrak{p}_z(D)$ is given by  \[ G_z u = \mathcal{F}^{-1}\mathfrak{p}_z^{-1}\mathcal{F}u =  \mathfrak{E}_z * u \] where $\mathfrak{E}_z = \mathcal{F}^{-1}(\mathfrak{p}_z^{-1})$. Here, $\mathcal{E}^\prime(\R^n)$ denotes the space of compactly supported distributions whereas $\mathcal{S}'(\R^n)$ denotes the space of tempered distributions. The key property of $\mathfrak{E}_z$, which we will employ to construct CGO solutions, is the following pointwise estimate
\begin{lemma}\label{lem:z_pointwise}
	There exists a constant $A > 0$ such that for all $z \in \C^3$ with $z \cdot z = 0$ it holds that
	\begin{equation}\label{eq:3d_pointwise_ez}
		\abs{\mathfrak{E}_z(x)} \leq \frac{A}{\abs{x}}\,, \quad (x \neq 0)\,.
	\end{equation}
\end{lemma}
We remark that by letting $z \to 0$, one obtains that $A \geq 1/(4\pi)$. The estimate \eqref{eq:3d_pointwise_ez} can be derived from \cite[Lemma 3.16]{shen2001absolutekato} by a change of coordinates. By employing a different method of proof than the one given in \cite{shen2001absolutekato}, it is possible to show that one may chose $A \leq 3\sqrt{2}/(4\pi)$. This is carried out in detail in Appendix \ref{sec:point}\,.

Furthermore, we require the following corollary of the classical  Sylvester-Uhlmann estimate \cite[Proposition 2.1]{sylvester1987global}.

\begin{lemma}\label{SU_estimate}
	There is a $C \neq C_z$ and a $s_0 > 0$ such that for all $u \in L^2(U)$ and $z \in \C^n$ with $\abs{z} > s_0$ we have
	\[ \norm{G_z u}_{L^2(U)} \leq C \abs{z}^{-1}\norm{u}_{L^2(U)}\,.
	\] 
\end{lemma}

\begin{proof}
	Since $U$ is bounded, there is an $R > 0$ such that $U$ is contained in some closed ball $B(R) = \{x \in \R^n: \abs{x} \leq R\}$. Suppose that $-1 < \delta < 1$ and define the weighted space
	\[
	L^2_\delta = L^2(\R^n, ( 1+\abs{x}^2 ) ^{\delta/2}\,\diff x )\,,
	\]
	with the norm defined by
	\[
	\norm{f}^2_{L^2_\delta} = \int\limits_{\R^n} \abs{f(x)}^2( 1+\abs{x}^2 ) ^{\delta/2}\,\diff x, \quad f \in L^2_\delta\,.
	\]
	In \cite[Proposition 2.1]{sylvester1987global} it is proven that for all $z$ fulfilling the above assumptions and $u \in L^2_{\delta +1}$  and all $-1 < \delta < 0$, we have
	\[
	\norm{G_z u}_{L^2_\delta} \leq C_\delta \abs{z}^{-1} \norm{u}_{L^2_{\delta+1}}\,.
	\]
	Every $u \in L^2(U)$ may be extended to an element of $L^2(\R^n)$ by setting it to $0$ outside of $u$. Denoting this extension again by $u$, we  deduce that there exists a constant $C_{\delta,R}$ such that for all $u \in L^2(U)$ we have
	\begin{align*}
		\norm{G_zu}_{L^2(U)} &\leq \norm{G_zu}_{L^2(B(R))} \\ &\leq C_{\delta,R}\norm{G_z u}_{L^2_\delta} \\ &\leq C_{\delta,R} \abs{z}^{-1} \norm{u}_{L^2_{\delta+1}} \\ &\leq C_{\delta,R}\abs{z}^{-1}\norm{u}_{L^2(U)}\,. \qedhere
	\end{align*}
\end{proof}
The following lemma is one of the key results that will enable the construction of CGO solutions. Here we employ the pointwise estimate \eqref{eq:3d_pointwise_ez}.
\begin{lemma}\label{main_lemma}
	Let $V,W \in \K_3^c$. There exists a $C > 0$ such that
	\begin{equation}\label{eq:interpolation}
		\norm{\abs{V}^{1/2}G_z\abs{W}^{1/2}f}_{L^2(\R^3)} \leq C\norm{V}^{1/2}_{\K_3}\norm{W}^{1/2}_{\K_3}\norm{f}_{L^2(\R^3)}, \quad (z \in \C^3, f \in L^2(\R^3))\,.
	\end{equation}
\end{lemma}
\begin{proof}
	Without loss of generality, we assume that $V,W \geq 0$. Let $S = \{\alpha \in \C: 0 < \real(\alpha) < 1\}$. We consider the operators $(T_z^\alpha)_{\alpha \in \overline S}$ defined on measurable step functions $f$ by
	\begin{equation*}
		T_z^\alpha f = V^{\alpha}\mathfrak{E}_z*(W^{1 - \alpha}f)\,.
	\end{equation*}
	Since $W^{\alpha}f$ is compactly supported and $\mathfrak{E}_z \in L^1_\loc(\R^3)$, it follows that $\mathfrak{E}_z *(W^{\alpha}f)$ is measurable and thus we have that $T_z^\alpha f$ is measurable as well. Therefore, $T_z^\alpha$ maps measurable step functions to measurable functions. Let us write $\alpha = t + \ii s$ with $0 \leq t \leq 1$ and $s \in \R$. It follows from \eqref{eq:3d_pointwise_ez} that
	\begin{equation}\label{eq:L1}
		\norm{T_z^{1 + \ii s} f}_{L^1(\R^3)} \leq A\int\limits_{\R^3} \frac{\abs{V(y)}}{\abs{y - x}}\abs{f(x)}\,\diff y\,\diff x \leq A \norm{V}_{\K_3} \norm{f}_{L^1(\R^3)}\,,
	\end{equation}
	and
	\begin{equation}\label{eq:Linfty}
		\norm{T_z^{\ii s} f}_{L^\infty(\R^3)} \leq A \esssup_{x \in \R^3}  \int\limits \frac{\abs{V(y)}}{\abs{y - x}}\,\diff y \leq A \norm{V}_{\K_3} \norm{f}_{L^\infty(\R^3)}\,.
	\end{equation}
	
	It is well-known that the Stein interpolation theorem can be employed to deduce \eqref{eq:interpolation} from \eqref{eq:L1} and \eqref{eq:Linfty}, see \cite[Theorem 2]{stein1956interpolation} for a detailed argument.
\end{proof}
Let $n \geq 3$, denote by $\Sigma$ the space of measurable step functions and by $L^0$ the space of measurable functions $\R^n \to \C$. The interpolation argument which we employed here remains valid if we replace the operator $G_z$ by an arbitrary integral operator $T:\Sigma \to L^0$ such that $Tf = k * f$ for some $k \in L^1_\loc(\R^n)$ obeying the inequality $\abs{k(x)} \leq A\abs{x}^{2 - n}$ for some constant $A$. 
Examples are given by the operators $((I - \Delta)^{-1})$ (given by convolution with the Bessel potential $(F_1(\cdot,1))$ or $\I_2 = (-\Delta)^{-1}$ (given by convolution with the Riesz potential $I_2$). We obtain the following proposition:
\begin{proposition}\label{prop:general}
	Let $V,W \in \K_n^c$. and $T:\Sigma \to L^0$ as above. It holds that
	\begin{equation*}
		\norm{\abs{V}^{1/2}T\abs{W}^{1/2}f}_{L^2(\R^n)} \leq C\norm{V}^{1/2}_{\K_n}\norm{W}^{1/2}_{\K_n}\norm{f}_{L^2(\R^n)}, \quad (f \in L^2(\R^n))\,.
	\end{equation*}
\end{proposition}
\begin{lemma}\label{symmetrized2}
	Let $V,W \in \K_3$ such that $\spt V \subseteq U$, $\spt W \subseteq U$, where $U \subseteq \R^3$ is open and bounded.  For all \( \delta > 0  \) there exists a \(s_0 > 0\) depending only on $V,W,\delta$ such that for all $z \in \C^3$ satisfying $\abs{z} \geq s_0$ and $z \cdot z = 0$, we have \[ \norm{\abs{V}^{1/2}G_z\abs{W}^{1/2}}_{L^2(\R^3) \to L^2(\R^3)} \leq \delta\,. \] 
\end{lemma}
\begin{proof}
	Let $\delta, \varepsilon > 0$. Again, we assume that $V,W \geq 0$ without loss of generality. We may write $V = V_1 + V_2$ with $V_1, V_2$ such that $V_1 \in \K_3$ with $\norm{V_1}_{\K_3} <  \varepsilon$ and $V_2 \in L^\infty(\R^3)$ with $\norm{V_2}_{\K_3} \leq \norm{V}_{\K_3}$. The same decomposition is made for $W$. It follows that
	\begin{equation}\label{eq:decoupling}
		\norm{V^{1/2}G_zW^{1/2} u}_{L^2(\R^3)}  \leq \sum_{i,j = 1}^2 \norm{V_i^{1/2}G_zW_{j}^{1/2}u}_{L^2(\R^3)}
	\end{equation}
	From Lemma \ref{main_lemma}, we obtain for all $u \in L^2(\R^3)$ such that $\norm{u}_{L^2(\R^3)} = 1$ the inequalities
	\begin{align*}
		\norm{V_1^{1/2}G_zW_1^{1/2}u}_{L^2(\R^3)} &\leq C\norm{V_1}^{1/2}_{\K_3} \norm{W_1}^{1/2}_{\K_3} \leq C\,\varepsilon \\
		\norm{V_1^{1/2}G_zW_2^{1/2}u}_{L^2(\R^3)} &\leq C\norm{V_1}^{1/2}_{\K_3} \norm{W}^{1/2}_{\K_3}  \leq C\norm{W}_{\K_3}\, \varepsilon^{1/2} \\
		\norm{V_2^{1/2}G_zW_1^{1/2}u}_{L^2(\R^3)} &\leq C\norm{V}^{1/2}_{\K_3} \norm{W_1}^{1/2}_{\K_3}  \leq C\norm{V}_{\K_3}\, \varepsilon^{1/2}\,.
	\end{align*}
	By the assumption on the supports of $V$ and $W$, we may apply Lemma \ref{SU_estimate} and conclude that there exists a $\tilde s_0$ such that for all $z \in \C^3$ with $z \cdot z = 0$ and $\abs{z} \geq \tilde s_0$, we have
	\[
	\norm{V^{1/2}_2G_zW^{1/2}_2u}_{L^2(\R^3)} \leq C\norm{V_2}^{1/2}_{L^\infty(\R^3)}\norm{W_2}^{1/2}_{L^\infty(\R^3)}\,\abs{z}^{-1}\,.
	\]
	Thus, we can make the right-hand side of \eqref{eq:decoupling} smaller than $\delta$ if we choose $\varepsilon$ small enough and $\abs{z}$ large enough. To be precise, assuming that $\delta  \leq 1$, we choose $\varepsilon < 1$ such that
	\[
	\varepsilon < \varepsilon^{1/2} < \frac{\delta}{6C\max\{ \norm{V}_{\K_3}, \norm{W}_{\K_3}  \}}
	\]
	and then we choose $s_0 = s_0(\varepsilon)$ such that
	\[
	s_0 > \max\bigg\{\frac{2C \max\{\norm{V_2}_{L^\infty(\R^3)}, \norm{W_2}_{L^\infty(\R^3)}\} }{\delta},\tilde s_0\bigg\} \,.
	\]
	For all $z \in \C^3$ such that $z \cdot z = 0$ and $\abs{z} > s_0$, it follows that
	\[
	\norm{vG_zwu}_{L^2(\R^3)} < \frac{\delta}{2} + \frac{\delta}{2} = \delta\,. \qedhere
	\]
\end{proof}
\begin{proof}[Proof of Theorem \ref{CGO_theorem}]
	We wish to solve the equation $(\mathfrak{p}_z(D) + V)r_z = -V$ with $r_z \in H^1(U)$ and $\norm{r_z}_{L^2(U)} \to 0$ as $z \to \infty $. Setting $v = \abs{V}^{1/2}$, $w = V/\abs{V}^{1/2}$, and setting $r_z = G_z(vf_z)$ with $f_z \in L^2(\R^3)$ to be determined, we obtain the equation
	\[
	vf_z + VG_zvf_z = -V\,.
	\]
	Observe that $r_z \in L^2_\loc(\R^3)$ since we have $\mathbf{1}_Kr_z = \mathbf{1}_K G_z(wf_z) \in L^2(\R^3)$ for any compact set $K$ by Lemma \ref{symmetrized2}. Since we have $V =vw$, we can find $f_z$ by solving
	\[
	f_z + wG_zvf_z = -v\,.
	\]
	By Lemma \ref{symmetrized2}, we can find a $s_0 > 0$ such that $\norm{wG_zv}_{L^2(\R^3) \to L^2(\R^3)}  \leq \frac{1}{2}$ for all $z \in \C^3$ with $\abs{z} > s_0$ and $z \cdot z = 0$. By a standard Neumann series argument we obtain that $I + wG_zv$ is invertible in $L^2(\R^3)$. Thus, we can find a unique solution $f_z$ to the above equation obeying the estimate
	\[
	\norm{f_z}_{L^2(\R^3)} \leq 2\norm{v}_{L^2(\R^3)}\,.
	\] It follows that $V r_z = v(wG_{z}vf_z) \in L^1(\R^3)$ as it is the product of two $L^2$-functions and that $(P_z(D) + V)r_z = -V$ in the distributional sense. By Corollary \ref{symmetrized2} (applied with $v$ replaced by $\mathbf{1}_U$), we have that $\norm{G_zv}_{L^2(\R^3) \to L^2(U)} \to 0$ as $\abs{z} \to \infty$ and thus
	\begin{align*}
		\norm{r_z}_{L^2(U)} &= \norm{G_z(vf_z)}_{L^2(U)} \\
		&\leq \norm{G_zv}_{L^2(\R^3) \to L^2(U)}\norm{f_z}_{L^2(\R^3)} \\
		&\leq 2\norm{G_zv}_{L^2(\R^3) \to L^2(U)}\norm{v}_{L^2(\R^3)} \to 0.\quad (z \to \infty)
	\end{align*}
	Setting $u_z = \e_z(1 + r_z)$, we obtain that
	\[
	(-\Delta + V)u_z = 0\,.
	\]
	It remains to show that $\nabla u_z \in L^2(U;\C^3)$. For this, consider $h_z = wu_z$. We have that
	\begin{equation*}
		h_z = w\e_z + wG_z(vf_z) \in L^2(\R^3)
	\end{equation*}
	and
	\begin{equation*}
		\Delta u_z = vwu_z = v h_z\,.
	\end{equation*}
	We note that $\F(\Delta u_z) = \F(v h_z)$ is smooth as it is the Fourier transform of a compactly supported function. It follows from the semigroup property of the Riesz potentials (Proposition \ref{prop: semigroup}) that \[ u_z = \I_2 \Delta u_z = \I_2(vh_z) = \I_1 \I_1(vh_z) \] and we deduce
	\begin{equation*}
		\nabla u_z = \nabla \I_2( v h_z) = \nabla \I_1 \I_1( v h_z)\,.
	\end{equation*}
	Since $\nabla \I_1 = -iD/\abs{D}$, it follows from Plancherel's theorem that $\nabla \I_1$ is a bounded operator $L^2(\R^3) \to L^2(\R^3;\C^3)$. Thus, since $v = \abs{V}^{1/2}$, it follows that
	\begin{align*}
		\norm{\nabla u_z}^2_{L^2(\R^3)} &\leq C\norm{\I_1 v h_z}_{L^2(\R^3)}^2 = (\I_1 vh_z, \I_1 vh_z)_{L^2(\R^3)} = (v\I_2 v h_z,h_z)_{L^2(\R^3)} \\
		&\leq C\norm{V}_{\K_3}\norm{h_z}^2_{L^2(\R^3)} < \infty\,,
	\end{align*}
	where in the last line, we have employed Proposition \ref{prop:general} with $T = \I_2$.
\end{proof}

\section{Definition of the Dirichlet-to-Neumann operator}\label{sec:DTN}
In this section, we discuss some operator-theoretic results which lead to the definition of the Dirichlet-to-Neumann operator for Kato pertubations of the Laplacian. We require some facts about closed quadratic forms which can be found in \cite{reed1980methods} and \cite{reed1975methods2} and some standard results on spectral theory taken from \cite{mclean2000strongly}. Let $H$ be a Hilbert space. We recall that a quadratic form $\a:\dom(\a) \to \R$ defined on some dense subspace $\dom(\a) \subseteq H$ is called \emph{semibounded} with lower bound $\gamma$ if there exists a $\gamma \in \R$ such that $\a(f) \geq \gamma \norm{f}_H^2$ holds for all $f \in \dom(\a)$. We write this as $\a \geq \gamma$. A semibounded quadratic form is called \emph{closed} if the set $\dom(\a)$ endowed with the inner product given by the quadratic form $\norm{\cdot}^2_\a = ((1-\gamma)\norm{\cdot}^2_H + \a)$ is a Hilbert space.

\begin{theorem}\label{KLMN_theorem} (\cite[Theorem X.17]{reed1975methods2})
	Let $H$ be a Hilbert space and $\a \geq 0$, $\b$ be real-valued quadratic forms on $H$ such that $\a$ is closed, $\dom(\a) \subseteq \dom(\b)$ and there exist constants $0 < \alpha < 1$, $\beta > 0$ for which the inequality
	\begin{equation*}
		\abs{\b(f)} \leq \alpha a(f) + \beta\norm{f}_H^2\, , \quad f \in \dom(\a)
	\end{equation*}
	holds. Then, the form $\a + \b$ given by
	\begin{equation*}
		\dom(\a + \b) = \dom(\a) \quad (\a + b)(f) = \a(f) + \b(f), \quad f \in \dom(\a)
	\end{equation*}
	is well-defined, closed and it holds that
	\begin{equation*}
		\a + \b \geq -\beta\,.
	\end{equation*}
	Furthermore, the norms $\norm{\cdot}_{\a + \b}$ and $\norm{\cdot}_{\a}$ are equivalent.
\end{theorem}
It is well known that that for each closed, real quadratic form $\a$, the operator $A$ associated to $\a$ given by
\begin{equation*}
	\text{Graph}(A) = \bigg\{ (f,g) \in \dom(\a) \times H:~ \forall h \in \dom(\a):~ \a(f,h) = (g,h) \bigg\}
\end{equation*}
is self-adjoint and bounded from below, see \cite[Theorem VIII.15]{reed1980methods}. Given two forms $\a$ and $\b$ we say that $\b$ is \emph{small} with respect to $\a$ if the conditions given in Theorem \ref{KLMN_theorem} are fulfilled. In this case, we may associate to $\a+\b$ a self-adjoint operator with lower bound $\beta$ which we denote by $A+B$. This is the so called \emph{form sum} of $A$ and $B$. For the form sum, it holds that $\dom(A +B) \subseteq \dom(\a + \b)$.

Let $U$ be an open set with Lipschitz boundary in $\R^n$, $n \geq 3$. Suppose that $V \in \K_n$ such that $\spt V \subseteq U$ and define the quadratic form $\q_V$ by
\begin{equation*}
	\dom(\q_V) = \bigg\{\,f \in L^2(U): \int\limits_U V\abs{f}^2\,\diff x \in \R\, \bigg\}, \quad \q_V(f) = \int\limits_U V\abs{f}^2\,\diff x\,.
\end{equation*} Moreover, we recall the definition of the classical Dirichlet form
\[
\dom(\d) = H^1_0(U), \quad \d(f) = \int\limits_U \abs{\nabla f}^2\,\diff x\,.
\]
\begin{lemma}
	The form $\q_V$ is small with respect to $\d$.
\end{lemma}
For a proof we refer to \cite[p. 459]{simon1982schrodinger}.

Thus, for any $V \in \K_n$, we may consider the form $\a_V = \d + \q_V$ with $\dom(\a_V) = H^1_0(U)$. The associated operator $A_V$ is the realization of the Schrödinger operator $-\Delta + V$ on $L^2(U)$ with Dirichlet boundary conditions.

To construct the Dirichlet-to-Neumann operator, we must first construct a bounded operator $P_V:H^{1/2}(\Gamma) \to H^1(U)$ that solves the corresponding Dirichlet problem. In this step, the assumption that $0 \notin \sigma(A_V)$ plays a role. Our approach is fairly standard and follows the one in \cite{krupchyk2016inverse}.

Since the embedding $H^1_0(U) \hookrightarrow L^2(U)$ is compact and we have $\dom(A_V) \subseteq H^1_0(U)$, it follows that the embedding $\dom(A_V) \hookrightarrow L^2(U)$ is compact as well. Therefore $\sigma(A_V)$ is a countable set consisting only of eigenvalues. By the spectral theorem for compact operators, we have the orthogonal decomposition
\begin{equation*}
	L^2(U) = \bigoplus\limits_{\lambda \in \sigma(A_V)} E_\lambda
\end{equation*}
with $E_\lambda$ denoting the eigenspace associated to the eigenvalue $\lambda$.

Denote by $\pi_\lambda$ the orthogonal projection onto $E_\lambda$. By form smallness and the Poincar\'{e} inequality, we have that
\begin{equation*}
	\abs{\a_V} \leq C(\d + \norm{\cdot}^2_{L^2(U)}) \leq C\d\,,
\end{equation*}
which implies that the linear operator
\begin{equation*}
	\mathfrak{A}_V: H^1_0(U) \to H^{-1}(U), \quad \dual{\mathfrak{A}_V f, g} = \a_V(f,g), \quad (f,g \in H^1_0(U)) 
\end{equation*}
is bounded.  This operator has the spectral decomposition
\begin{equation*}
	\mathfrak{A}_V = \sum_{\lambda \in \sigma(A_v)} \lambda\pi_\lambda
\end{equation*}
with the sum converging in the strong operator topology of $\L(H^1_0(U),H^{-1}(U))$. According to \cite[Theorem 2.37]{mclean2000strongly}, if $0 \notin \sigma(A_V)$ then it follows that $\mathfrak{A}_V$ has a bounded inverse given by
\begin{equation*}
	(\mathfrak{A}_V)^{-1} = \sum_{\lambda \in \sigma(A_v)} \lambda^{-1}\pi_\lambda\,.
\end{equation*}

With the previous result, we are now in a position to define the Dirichlet-to-Neu-mann operator as follows: For each $f \in H^{-1}(U)$, there exists a unique $u = (\mathfrak{A}_V)^{-1}f \in H^1_0(U)$ such that $(-\Delta + V)u = f$ holds in the weak sense. Furthermore, given $\phi \in H^{1/2}(\Gamma)$, we may then find a $f_\phi \in H^1(U)$ such that $f_\phi|_{\Gamma} = \phi$ where the restriction is to be understood in the sense of the trace operator $\gamma: H^1(U) \to H^{1/2}(\Gamma)$. Noting that $(-\Delta + V)f_\phi$ as an element of $\D^\prime(U)$ is given by
\[
\dual{(-\Delta + V)f_\phi, h} = \a_V(f_\phi,h), \quad (h \in \D(U))
\]
we obtain that $(-\Delta + V)f_\phi \in H^{-1}(U)$ since $\D(U)$ is dense in $H^1_0(U)$ by definition. Setting
\begin{equation*}
	P_V\phi = -(\mathfrak{A}_V)^{-1}(-\Delta + V)f_\phi + f_\phi\, \quad \phi \in H^1(U)\,,
\end{equation*}
we obtain a linear, bounded operator $P_V:H^{1/2}(\Gamma) \to H^1(U)$ that is well-defined since if $f$ and $g$ are two functions such that $f|_{\Gamma} =g|_{\Gamma}$, it follows that $f-g \in H^1_0(U)$ and thus that
\begin{equation*}
	(\mathfrak{A}_V)^{-1}(-\Delta + V)(f-g)  = f-g\,.
\end{equation*}
It is clear that $P_V\phi|_{\Gamma} = \phi$ by construction. Moreover
\[
(-\Delta + V)P_V\phi = -(-\Delta + V)f_\phi + (-\Delta + V)f_\phi = 0\,,
\]
so $P_V\phi$ solves the Dirichlet problem \eqref{eq:Dirichlet_problem}.

We may therefore define the Dirichlet-to-Neumann operator as follows:
\begin{definition}\label{def:dtn}
	Let $V \in \K_n$ with $\spt V \subseteq U \subseteq \R^n$ where $U$ is open and bounded with Lipschitz boundary. Suppose that $0 \notin \sigma(A_V)$ where $A_V$ is the realization of the Schrödinger operator $-\Delta + V$ on $L^2(U)$ with Dirichlet boundary conditions. Denote by \[P_V:H^{1/2}(\Gamma) \to H^1(U)\] the solution operator for the Dirichlet problem \eqref{eq:Dirichlet_problem}. We define the Dirichlet-to-Neumann operator $\Lambda_V:H^{1/2}(\Gamma) \to H^{-1/2}(\Gamma)$ by the identity
	\begin{equation*}
		\dual{\Lambda_V\phi,\psi}_{H^{-1/2}(\Gamma) \times H^{1/2}(\Gamma) } = \int\limits_U \nabla P_V\phi \cdot \nabla P_V\psi + VP_V\phi  {P_V\psi}\,\diff x\,.
	\end{equation*}
\end{definition} 

We note the following important consequence of this definition: If $u$ is \emph{any} element of $H^1(U)$ such that $u|_{\Gamma} = \psi$, then
\begin{equation}\label{eq:dtn_def}
	\dual{\Lambda_V\phi,\psi}_{H^{-1/2}(\Gamma) \times H^{1/2}(\Gamma)} = \int\limits_U \nabla P_V\phi \cdot \nabla u + VP_V\phi  {u}\,\diff x\,.
\end{equation}
This is true because if we set $v =u - P_V{\phi}$, then $(-\Delta + V)v = (-\Delta + V)u$, therefore
\begin{equation*}
	\int\limits_U \nabla P_V\phi \cdot \nabla u + VP_V\phi  {u}\,\diff x -	\dual{\Lambda_V\phi,\psi} = \int\limits_U \nabla P_V\phi \cdot \nabla v + VP_V\phi  {v}\,\diff x\,.
\end{equation*}
Since $-\Delta P_V\phi + VP_V\phi = 0$ and $v|_\Gamma = 0$, Eq.~\eqref{eq:dtn_def} follows from integration by parts.

\section{Proof of Theorem \ref{main_theorem}}\label{sec:main}

\begin{proof}[Proof of Theorem \ref{main_theorem}]
	By Eq.~\eqref{eq:dtn_def} and the assumption $\Lambda_{V_1} = \Lambda_{V_2}$, we have for all $\phi_1,\phi_2 \in H^{1/2}(\Gamma)$ the equation
	\[
	0 =\dual{\Lambda_{V_1}\phi_1,\phi_2} - \dual{\Lambda_{V_2}\phi_2,\phi_1} = \int\limits_{U}(V_1 - V_2)P_{V_1}\phi_1{P_{V_2}\phi_2}\,\diff x\,.
	\]
	Thus, setting $V = V_1 - V_2$ we obtain
	\begin{equation}\label{density_equation}
		0 = \int\limits_{U}Vu_1 u_2\,\diff x
	\end{equation}
	for all $u_1,u_2 \in H^1(U)$ such that $-\Delta u_j + V_ju_j = 0$. For each $\xi \in \R^3$, we will choose $u_1$ and $u_2$ such that $u_1(x)u_2(x) \approx \euler^{\ii \xi \cdot x}$. The theorem will then follow from Fourier inversion.
	
	Let $\xi \in \R^3$. We choose unit vectors $\eta^1,\eta^2 \in \R^3$ such that the set $\{\xi,\eta^1,\eta^2\}$ is orthogonal. For each $s > \max\{s_0,\abs{\xi}^2/4\}$, we define
	\begin{align*}
		z^1 =z_s^1 &= \frac{\xi}{2} + r\eta^1 + \ii s\eta^2 \\
		z^2 =z_s^2 &= \frac{\xi}{2} - r\eta^1 - \ii s\eta^2
	\end{align*}
	where $r > 0$ is chosen such that $\abs{\xi}^2/4 + r^2 = s^2$. By construction, it holds that \[\abs{z_s^1}^2 = \abs{z_s^2}^2 = 2s^2\,.\] We have $z^j \cdot z^j = 0$ for $j=1,2$ and $z^1 + z^2 = \xi$. From Theorem \ref{CGO_theorem}, we obtain for $j = 1,2$ elements $u_{z^j} \in H^1(U)$ satisfying $-\Delta u_{z^j} + V_ju_{z^j} = 0$ of the form
	\[
	u_{z^j} = \e_{z^j}(1 + r_{z^j})
	\]
	with $\norm{r_{z^j}}_{L^2(U)} \to 0$ as $s \to \infty$.
	
	Inserting $u_{z^1}$ and $u_{z^2}$ into equation \eqref{density_equation}, it follows that
	\begin{align*}
		0 &= \int\limits_U Vu_{z^1} u_{z^2}\,\diff x \\ &= \int\limits_U V\e_{\xi}\,\diff x + \int\limits_U V\e_\xi(r_{z^1} +  r_{z^2} + r_{z^1} r_{z^2})\,\diff x \,.
	\end{align*}
	We now show that the second integral goes to $0$ as $s \to \infty$. Defining $v = \abs{V}^{1/2}$, we obtain from the proof of Theorem \ref{CGO_theorem} that there exist $f_{z^1}, f_{z^2} \in L^2(U)$ such that
	\begin{equation}\label{eq:final_proof}
		r_{z^j} = G_{z^j}(vf_{z^j}), \quad \norm{f_{z^j}} \leq 2\norm{v}_{L^2} = 2\norm{V}_{L^1}, \quad j=1,2.
	\end{equation}
	
	We set $w = V/\abs{V}^{1/2}$. Then, we get from Corollary \ref{symmetrized2} and \eqref{eq:final_proof} the estimate
	\begin{align*}
		\biggabs{\int\limits_{\R^3} V\e_\xi r_{z^1} r_{z^2}} &= \biggabs{\int\limits_{\R^3} \e_\xi vr_{z^1} wr_{z^2}} \\ &\leq \norm{vG_{z^1}(vf_{z^1})}_{L^2(\R^3)} \norm{wG_{z^2}(vf_{z^2})}_{L^2(\R^3)} \\ &\leq C\norm{vG_{z^1}v}_{L^2(\R^3) \to L^2(\R^3)} \norm{wG_{z^2}v}_{L^2(\R^3) \to L^2(\R^3)} \norm{V}_{L^1(\R^3)}^2 \\ 
	\end{align*}
	It is clear that the right hand side goes to $0$ as $s \to \infty$. The remaining integrals over $V\e_\xi r_{z^1}$ and $Ve_\xi r_{z^2}$ are treated in the same fashion.
	
	We obtain \[0 = \int\limits_{\R^3} V\e_\xi\,\diff x \quad (\xi \in \R^3)\,. \] But since we have $V \in L^1(\R^3)$, this means that $V = 0$ by Fourier inversion and thus \[V_1 = V_2\,. \qedhere \]
\end{proof}

\appendix

\section{Bessel and Riesz potentials}\label{sec:bessel}

Let $d \geq 2$. We write $\S(\R^d)$ for the Schwartz space and $\S^\prime(\R^d)$ for the space of tempered distributions. For $\mu \in \C \setminus (-\infty,0]$ and $\alpha \in \C \setminus \{0,-1,-2,\dots\}$ we define the \emph{Bessel potentials} ${F_\alpha(\cdot,\mu) \in \mathcal{S}'(\R^d)}$ by
\begin{equation} \label{eq: Bessel_definition}
	F_\alpha(x,\mu) = \frac{\Gamma(\alpha)}{(2\pi)^{d}}\int\limits_{\R^{d}} \euler^{\ii x\cdot\xi} [ \abs{\xi}^2 + \mu ]^{-\alpha}\,\diff\xi.
\end{equation}
These are well defined elements of $\mathcal{S}'(\R^{d})$ since for the above choices of $\alpha$ and $\mu$, the function $\xi \mapsto (\xi^2 + \mu)^{-\alpha}$ is locally integrable and grows at most polynomially.

We observe that the Bessel potentials are radial functions in $x$, therefore we may abuse notation and set $r = \abs{x}$ as well as $F_\alpha(x,\mu) = F_\alpha(r,\mu)$. Taking Fourier transforms, it is easy to see that 
\begin{equation}\label{eq:fundamental1}
	(-\Delta + \mu)F_1(\cdot,\mu) = \delta_0
\end{equation}
and
\begin{equation}\label{eq:fundamental2}
	(-\Delta + \mu)F_\alpha(\cdot,\mu) = (\alpha - 1)F_{\alpha -1}(\cdot,\mu)\, \quad \alpha \neq 1.
\end{equation}
We have the well-known explicit representation (see for example \cite{kenig1987uniform})
\begin{equation}\label{eq:bessel_potential}
	F_\alpha(r,\mu) = \frac{2^{-\alpha+1}}{(2\pi)^{d/2}} \bigg( \frac{\mu^{1/2}}{r} \bigg)^{d/2 - \alpha}K_{d/2 - \alpha}(\mu^{1/2} r)\,,
\end{equation}
where for $\lambda \in \C$ we denote by $K_\lambda$ the modified Bessel function of the second kind (also known as Bessel function of the third kind) which is given for complex numbers $w \in \C$ such that $\real(w) > 0$ by
\[
K_\lambda(w) = \int\limits_0^\infty \euler^{-w\cosh(t)}\cosh(\lambda t)\,\diff t\,.
\]

The following well-known properties of these functions can be found e.g. in \cite{kenig1987uniform}:
\begin{align}
	\abs{\euler^{\lambda^2} \lambda K_\lambda(\mu^{1/2} r)} &\leq C_{\real \lambda} \abs{\mu}^{-\real \lambda/2 }r^{-\real \lambda } &\text{if} \quad 0 < \abs{\mu}^{1/2}r \leq 1, \label{eq:short_range_bessel}\\ K_\lambda(\mu^{1/2} r) &= a_\lambda(r,\mu)\mu^{-1/4}r^{-1/2}\euler^{-\mu^{1/2} r} &\text{if} \quad 0 < \abs{\mu}^{1/2}r < \infty \label{eq:long_range_bessel}\,, 	
\end{align}
where the functions $a_\lambda(r,\mu)$ are given by
\begin{equation*}
	\Gamma(\lambda + 1/2) a_\lambda(r,\mu) = \bigg(\frac{\pi}{2}\bigg)^{1/2} \int\limits_0^\infty \euler^{-t}t^{\lambda - 1/2} \bigg( 1 + \frac{t}{\mu^{1/2} r} \bigg)^{\lambda - 1/2}\,\diff t \label{eq:symbol_formula}\,.
\end{equation*}
Since
\begin{equation*}
	\abs{ a_\lambda(r,\mu)} \leq C_{\real(\lambda)} \quad (\abs{\mu}^{1/2}r \geq 1)\,,
\end{equation*}
it follows that
\begin{equation}\label{eq:long_range_bessel3}
	\abs{K_\lambda(\mu^{1/2} r)} \leq C_{\real(\lambda)}\abs{\mu}^{-1/4}r^{-1/2}\euler^{-\real \mu^{1/2} r} \quad (\abs{\mu}^{1/2}r \geq 1)\,.
\end{equation}

From \eqref{eq:short_range_bessel} and \eqref{eq:long_range_bessel3}, we see that for $\real(\alpha) \leq d/2$ and $\alpha \neq d/2$, the estimates
\begin{align}
	\abs{F_\alpha(r,\mu)} &\leq \frac{C_{\real \alpha}\euler^{(\imag \alpha)^2}}{r^{d - 2\real(\alpha)}} &(\abs{\mu}^{1/2}r \leq 1) \label{eq:short_range_bessel_potential} \\
	\abs{F_\alpha(r,\mu)} &\leq C_{\real \alpha}\euler^{(\imag \alpha)^2}  \frac{\abs{\mu}^{d/4 -\real(\alpha)/2}}{r^{d/2 - \real(\alpha)}}\frac{\euler^{-\real \mu^{1/2} r}}{\abs{\mu}^{1/4}r^{1/2}}&(\abs{\mu}^{1/2}r \geq 1) \label{eq:long_range_bessel_potential} 
\end{align}
hold. Regarding the case where $\alpha = d/2$, we note that from \eqref{eq:long_range_bessel3} it follows that the estimate \eqref{eq:long_range_bessel_potential} remains valid while we have that
\[
\abs{K_0(\mu^{1/2}r)} \leq C_d\log(2/(\abs{\mu}^{1/2}r)) \quad (\abs{\mu}^{1/2} r \leq 1)
\]
and therefore
\begin{equation}\label{eq:short_range_bessel_potential_log}
	\abs{F_{d/2}(r,\mu)} \leq C_d{\log(2/(\abs{\mu}^{1/2}r))} \quad ((\abs{\mu}^{1/2} r \leq 1))\,. 
\end{equation}

We point out that in the particular case where $d=2$ and $\alpha = 1$, it follows from the identity \eqref{eq:bessel_potential} that
\begin{equation}\label{eq:bessel_2d}
	F_1(r,\mu) = \frac{1}{2\pi} K_0(\mu^{1/2}r)\,.
\end{equation}

We also employ the \emph{Riesz potentials} given by
\begin{equation*}
	I_\alpha(x) =  \frac{1}{(2\pi)^{d}}\int\limits_{\R^{d}} \euler^{\ii x\xi}\abs{\xi}^{-\alpha}\,\diff \xi = \pi^{-d/2}2^{-\alpha}\frac{\Gamma((d-\alpha)/2)}{\Gamma(\alpha/2)}\abs{x}^{\alpha -d}\,\quad 0 < \real \alpha < d\,.
\end{equation*}
We write $\I_\alpha f = I_\alpha * f$ for $f \in \mathcal{E}^\prime(\R^d)$. We have $\I_2(-\Delta f) = (-\Delta)\I_2 f = f$. Suppose now, that ${f \in \S^\prime(\R^d)}$ is such that $\F f \in C^\infty(\R^d \setminus \{0\})$ and $\abs{\cdot}^{-\alpha}\F f \in L^1_\loc(\R^d)$.
It follows that ${\abs{\cdot}^{-\alpha}\F f \in \S^\prime(\R^d)}$ and we may define
\begin{equation*}
	\I_\alpha f = \F^{-1}(\abs{\cdot}^{-\alpha}\F f)\,.
\end{equation*}
This definition is consistent with the previous one since if $f \in \mathcal{E}^\prime(\R^d)$, then
\begin{equation*}
	\F^{-1}(\abs{\cdot}^{-\alpha}\F f) = I_\alpha * f\,.
\end{equation*}
\begin{proposition}\label{prop: semigroup}
	We have the semigroup identity $\I_\alpha\I_\beta f= \I_{\alpha + \beta}f$ for all $f \in \mathcal{E}^\prime(\R^d)$ with $\abs{\cdot}^{-\alpha -\beta}\F f \in L^1_\loc$.
\end{proposition}
\begin{proof}
	Clearly, $\abs{\cdot}^{-\alpha}\F \I_\beta f \in L^1_\loc$. Thus, $\I_\alpha(\I_\beta f) \in \S^\prime(\R^d)$ is well-defined with Fourier transform $\abs{\cdot}^{-\alpha -\beta}\F f$.
\end{proof}

\section{Fundamental solution for the conjugated Laplacian}\label{sec:point}

Let $\tau \in \R$ and $n \geq 3$. We define the negative conjugated Laplacian $-\Delta_\tau$ by the formula
\[
-\Delta_\tau f = -\euler^{\tau x_1}\Delta (\euler^{-\tau x_1}f) = (-\Delta +2\tau\partial_1 - \tau^2)f, \quad (f \in \S^\prime(\R^n))\,.
\]
It is a Fourier multiplier with symbol $p_\tau(\xi) = \abs{\xi}^2 + 2 \ii \tau\xi_1 - \tau^2 = \abs{\xi^\prime}^2 + (\xi_1 + \ii \tau )^2$ where we have written $\xi = (\xi_1,\xi^\prime) \in \R \times \R^{n-1}$.

We will show that $\F^{-1}p_\tau^{-1}$ may be expressed as a partial Fourier transform of a Bessel potential. In view of possible generalisations of our approach to higher dimensions (as discussed in Appendix \ref{sec:extensions}), we will work in a slightly more general setting than it is strictly necessary in the three-dimensional setting and consider $p_\tau^{-\alpha}$ for $\real(\alpha) < 2$.

To motivate our construction, we note that, by a formal application of Fubini's theorem, we have the identity
\begin{align*}
	\int\limits_{\R^n} \frac{\euler^{\ii x \xi}\,\diff \xi}{p_\tau(\xi)^\alpha} &= 	\int\limits_{-\infty}^\infty \int\limits_{\R^{n-1}} \frac{\euler^{\ii x^\prime \xi^\prime}\,\diff \xi}{[\abs{\xi^\prime}^2 + (\xi_1 + \ii \tau )^2]^\alpha} \,\diff \xi^\prime\, \euler^{\ii x_1 \xi_1} \diff \xi_1 \\ &= \int\limits_{-\infty}^\infty F_\alpha(x^\prime, (\xi_1 + \ii \tau)^2)\euler^{\ii x_1 \xi_1} \diff \xi_1\,.
\end{align*}
However, we have to be careful because both $p_\tau^{-\alpha}$ and $F_\alpha(x^\prime, (\xi_1 + \ii \tau)^2)$ have to be interpreted as tempered distributions. Our immediate goal will therefore be to justify this formal calculation in a rigorous manner.

To this end, let $0 < \real(\alpha) \leq (n-1)/2$ and consider the function
\begin{equation*}
	\Phi_\tau^\alpha: \R \times (\R^{n-1} \setminus \{0\}) \to \C \quad  (\xi_1,x') \mapsto F_\alpha(x',(\xi_1 + \ii \tau)^2)
\end{equation*}
We note that for $\xi_1 \neq 0$, it follows from \eqref{eq:bessel_potential} that
\begin{equation} \label{eq:Phi}
	\Phi_\tau^\alpha(\xi_1,x') = \frac{2^{-\alpha+1}}{(2\pi)^{\frac{n-1}{2}}} \bigg( \frac{\sqrt{(\xi_1 + \ii \tau)^2}}{\abs{x'}} \bigg)^{\frac{n-1}{2} - \alpha}K_{\frac{n-1}{2} - \alpha}\bigg(\sqrt{(\xi_1 + \ii \tau)^2}\abs{x'}\bigg)\,.
\end{equation}
Thus, if $\xi_1 \neq 0$, we may also write
\begin{align*} 
	\Phi_\tau^\alpha(\xi_1,x') = &\frac{2^{-\alpha+1}  }{(2\pi)^{\frac{n-1}{2}}} \bigg( \frac{\sgn (\xi_1)(\xi_1 + \ii \tau)}{\abs{x'}} \bigg)^{\frac{n-1}{2} - \alpha} \\ &\times K_{\frac{n-1}{2} - \alpha}\bigg(\sgn (\xi_1)(\xi_1 + \ii \tau)\abs{x'}\bigg)\,,
\end{align*}
with $\sgn(\xi_1) = \xi_1/\abs{\xi_1}$.
From \eqref{eq:short_range_bessel_potential} and \eqref{eq:long_range_bessel_potential} it follows that
\begin{align*}
	\abs{\Phi_\tau^\alpha(\xi_1,x')} &\leq C_{\real(\alpha)}\euler^{\imag(\alpha)^2} \\
	&\quad \times \max\bigg\{\abs{x'}^{2\real\alpha -(n-1)}, \abs{\xi_1 + \ii \tau}^{\frac{n-1}{4}- \frac{\real(\alpha)}{2}}\abs{x'}^{\real(\alpha) - \frac{n-1}{2}}\euler^{-\abs{\xi_1}\abs{x'}}\bigg\} \\
	&\leq C_{\real(\alpha)}\euler^{\imag(\alpha)^2}(\abs{\xi_1} + \tau +1)^{\frac{n-1}{4}- \frac{\real(\alpha)}{2}}\max\{\abs{x'}^{2\real\alpha -(n-1)},1\} \label{eq:rough}
\end{align*}
for $\alpha \neq \frac{n-1}{2}$, which shows that $\Phi_\tau^\alpha \in L^1_\loc(\R^n)$ and that $\Phi_\tau^\alpha$ has at most polynomial growth. The same is seen to be true for   $\Phi_\tau^{\frac{n-1}{2}}$ by applying $\eqref{eq:short_range_bessel_potential_log}$. 
Thus, $\Phi_\tau^\alpha$ acts as a tempered distribution defined by
\begin{equation*}
	\dual{\Phi_\tau^\alpha, \phi} = \int\limits_{\R^n} \Phi_\tau^\alpha(z)\phi(z)\,\diff z \quad (\phi \in \S(\R^n))\,.
\end{equation*}

We consider the partial inverse Fourier transform $\mathcal{F}_1^{-1}$, defined on Schwartz functions via
\[
\mathcal{F}_1^{-1}\phi(x_1,x') = \frac{1}{2\pi}\int\limits_{-\infty}^\infty \euler^{\ii x_1\xi_1}\phi(x',\xi_1)\diff \xi_1
\]
and extended to $\mathcal{S}^\prime(\R^n)$ by duality. Motivated by our formal calculation, we set
\begin{equation} \label{eq: fundamental_solution_definition}
	E_\tau^\alpha = \mathcal{F}_1^{-1}\Phi_\tau^\alpha,
\end{equation}
With this definition, the distribution $E^\alpha_\tau$ is indeed the inverse Fourier transform of the symbol $\Gamma(\alpha) p_\tau^{-\alpha}$:
\begin{proposition} \label{prop_fundamentalsolution_fourier}
	We have
	\begin{equation*}
		E^\alpha_\tau = \Gamma(\alpha)\mathcal{F}^{-1}p_\tau^{-\alpha} \quad (0 < \real(\alpha) < 2)
	\end{equation*}
	as an identity between tempered distributions.
\end{proposition}
\begin{proof}
	We need to show that
	\begin{equation}\label{eq:L1loc}
		p_\tau^{-\alpha} \in L^1_\loc, \quad \real(\alpha) < 2\,.
	\end{equation}
	If this is true, we may interpret $p_\tau^{-\alpha}$ as a tempered distribution. The identity then follows by applying both sides to Schwartz functions.
	
	Indeed, assume that \eqref{eq:L1loc} holds, let $\psi \in \mathcal{S}(\R \times \R^{n-1})$. We write
	\begin{equation*}
		(\mathcal{F}')^{-1}\psi(\xi_1,\xi') = \frac{1}{(2\pi)^{n-1}}\int\limits_{\R^{n-1}} \euler^{\ii x'\cdot\xi'}\psi(\xi_1,x')\,\diff x'.
	\end{equation*}
	Again, this operation extends to $\mathcal{S}'(\R^n)$ by duality. If \eqref{eq:L1loc} holds, then it follows for all $\phi \in \mathcal{S}(\R^n)$ that $\phi p_\tau^{-\alpha} \in L^1(\R^n)$. We then obtain from Fubini's theorem that
	\begin{align*}
		\Gamma(\alpha)\int\limits_{\R^n} \frac{\mathcal{F}^{-1}\phi(\xi)}{p_\tau^\alpha(\xi)}\,\diff \xi
		&= \Gamma(\alpha)\int\limits_{\R^{n-1}} \int\limits_{\R} \frac{(\mathcal{F}')^{-1}\mathcal{F}_1^{-1}\phi(\xi_1,\xi')}{p_\tau^{\alpha}(\xi_1,\xi')}\,\diff \xi_1\diff \xi' \\
		&= \Gamma(\alpha)\int\limits_{\R} \bigg( \int\limits_{\R^{n-1}}  \frac{(\mathcal{F}')^{-1}\mathcal{F}_1^{-1}\phi(\xi_1,\xi')}{p_\tau^{\alpha}(\xi_1,\xi')} \diff \xi' \bigg) \diff \xi_1 \\
		&=  \int\limits_{\R} \int\limits_{\R^{n-1}} F_\alpha\big(x',(\xi_1 + \ii \tau)^2\big)\mathcal{F}_1^{-1}\phi(\xi_1,x')\,\diff x'\diff \xi_1 \\
		&= \dual{\Phi_\tau^\alpha, \F_1^{-1}\phi} = \dual{E_\tau^\alpha, \phi}\,,
	\end{align*}
	which is what we wish to prove.
	
	Thus, we need to show \eqref{eq:L1loc}. The proof of this fact is standard in the case where $\alpha = 1$ and can be obtained by a flattening argument such as the one found in \cite{sylvester1987global}. We have included it here for the reader's convenience. The equality
	\begin{equation*}
		p_\tau(\xi)^{-\alpha} = \tau^{-2\alpha}p_1(\xi/\tau)^{-\alpha}
	\end{equation*}
	allows us to assume that $\tau = 1$ without loss of generality. Considering the integral
	\[
	\int\limits_B \abs{p_1(\xi)^{-\alpha}}\,\diff \xi\,
	\]
	as $B$ ranges over all balls of radius $r$ in $\R^n$, it is clear that the worst case occurs when $B$ is a ball centered on the set $\Sigma_1$. By applying a suitable rotation, we may assume that $B$ is a ball of radius $r$ centered at $\tau e_2$. We may also assume that $r < 1/2$, since we can split the integral into multiple integrals over smaller balls otherwise. Under these assumptions, the mapping $T(\xi) = \eta$ with $\eta$ given by
	\[
	\eta_2 = \frac{1}{2}(\abs{\xi}^2 - 1), \quad \eta_j = \xi_j ~\text{for}~ j \neq 2
	\]
	is a diffeomorphism of $B$ onto $T(B)$. In fact, we obtain that
	\[
	\det{T'(\xi)} = \det{ \begin{bmatrix}
			1 & 0 & 0 &\dots & 0 \\
			\xi_1 & \xi_2 & \xi_3 &\dots & \xi_n \\
			0 & 0 & 1 &\dots & 0 \\
			\vdots & \vdots & \vdots & \ddots & \vdots \\
			0 & 0 & 0 & \dots & 1
	\end{bmatrix} } = \xi_2
	\]
	from which it follows that
	\[
	\abs{\det T'(\xi)} \approx 1\quad (\xi \in B)\,,
	\]
	and thus
	\begin{align*}
		\int\limits_B \abs{p_1(\xi)^{-\alpha}}\,\diff \xi &\leq C \int\limits_{T(B)} \abs{(2\ii \eta_1 + 2\eta_2)^{-\alpha}}\,\diff\eta_1\diff\eta_2 \\ &\leq C_{\imag(\alpha)} \int\limits_{T(B)}(\eta_1^2 + \eta_2^2)^{-\real(\alpha)/2}\,\diff \eta_1\diff \eta_2 < \infty
	\end{align*}
	if $\real(\alpha) < 2$.
\end{proof}

Specializing to the case $\alpha = 1$, we shall write $E_\tau = E_\tau^1$ in the sequel. We observe that $E_\tau \in C^\infty(\R^n \setminus \{0\})$ since it is a fundamental solution of the elliptic operator $p_\tau(D)$. We consider the function
\[
H_\tau(x) = \frac{1}{2\pi}\int\limits_{-\infty}^\infty \euler^{\ii x_1\xi_1}\Phi_\tau^1(\xi_1,x^\prime)\,\diff\xi_1\,.
\]
The integral on the right hand side is absolutely convergent since $\Phi_\tau^1(\xi_1,x^\prime)$ is bounded by an integrable function if $\abs{\xi_1 + i}\abs{x^\prime} \leq 1$ by \eqref{eq:short_range_bessel_potential} and decays exponentially in $\xi_1$ for $\abs{\xi_1 + i}\abs{x^\prime} \geq1$ by \eqref{eq:long_range_bessel_potential}. By differentiating under the integral, we obtain that $H_\tau \in C^\infty(\{x^\prime \neq 0\})$. We claim that $H_\tau(x) = E_\tau(x)$ in this range. Indeed, taking $\phi \in \D(\{x^\prime \neq 0\})$, we obtain by Fubini's theorem
\[
\int\limits_{\R^n} H_\tau \phi\,\diff x = \int\limits_{\R^n} \Phi_\tau^1 \F_1^{-1}\phi\,\diff x = \int\limits_{\R^n} E_\tau \phi \,\diff x
\]
so it follows that $H_\tau = E_\tau|_{\{x^\prime \neq 0\}}$.

We are now in a position to prove the principal result of this section:
\begin{lemma}\label{3d_pointwise}
	For all $\tau \in \R$ and $x \in \R^3 \setminus \{0\}$ we have
	\begin{equation}\label{eq:3d_pointwise}
		\abs{E_\tau(x)} \leq \frac{3\sqrt{2}}{4\pi\abs{x}}\,.
	\end{equation}
\end{lemma}

\begin{proof}
	From the identities
	\begin{equation*}
		p_\tau(\xi)^{-1} = \tau^{-2}p_1(\xi/\tau)^{-1}
	\end{equation*}
	and Proposition \ref{prop_fundamentalsolution_fourier}, it follows that
	\begin{equation*}
		E_\tau(x) = \tau E_1(\tau x)
	\end{equation*}
	by the scaling property of the Fourier transform.
	This shows that we can set $\tau = 1$ without loss of generality since if \eqref{eq:3d_pointwise} holds for $\tau = 1$, it follows that
	\begin{equation*}
		\abs{E^1_\tau(x)} = \abs{\tau E^1_1(\tau x)} \leq \frac{\abs{\tau}3\sqrt{2}}{4\pi\abs{\tau x}} = \frac{3\sqrt{2}}{4\pi\abs{x}}\,.
	\end{equation*}
	
	By continuity, we may assume that $\abs{x^\prime} \neq 0$. Writing $r = \abs{x^\prime}$ and using \eqref{eq:bessel_2d}, we obtain
	\begin{align*}
		E_1(x_1,r) &= \frac{1}{4\pi^2}\int\limits_{-\infty}^{\infty} \euler^{\ii x_1\xi_1}K_0\bigg(r\sqrt{(\xi_1 + \ii )^2}\bigg)\,\diff \xi_1 \\
		&= \frac{1}{4\pi^2} \bigg(\,\int\limits_{0}^{\infty}\euler^{\ii x_1\xi_1}K_0(r(\xi_1+\ii ))\,\diff \xi_1 + \int\limits_{-\infty}^{0}\euler^{\ii x_1\xi_1}K_0(r(-\xi_1-\ii ))\,\diff \xi_1 \bigg) \\
		&= \frac{1}{4\pi^2}\bigg(\int\limits_{0}^{\infty}\euler^{\ii x_1\xi_1}K_0(r(\xi_1+\ii ))\,\diff \xi_1 + \int\limits_{0}^{\infty}\euler^{-\ii x_1\xi_1}K_0(r(\xi_1-\ii ))\,\diff \xi_1\bigg)\,. \\
	\end{align*}
	Since we have that
	\begin{equation*}
		K_0(\overline w) = \int\limits_0^\infty \euler^{- \overline w\cosh(t)}\,\diff t = \int\limits_0^\infty \overline {\euler^{-w\cosh(t)}} \,\diff t = \overline {K_0(w)} 
	\end{equation*}
	it follows that
	\begin{equation*}
		E_1(x_1,r) = \frac{1}{2\pi^2}\real \int\limits_{0}^{\infty}\euler^{\ii x_1\xi_1}K_0(r(\xi_1+\ii ))\,\diff \xi_1\,.
	\end{equation*}
	We observe that by the substitution $\xi_1 \mapsto r\xi_1$, we have
	\begin{equation}\label{eq:rest}
		\begin{aligned}
			\abs{ E_1(x_1,r)} &\leq \frac{1}{2\pi^2 r}\abs{\real \int\limits_{0}^{\infty}\euler^{\ii (x_1/r)\xi_1}K_0(\xi_1+\ii r)\,\diff \xi_1 } \\ &\leq \frac{1}{2\pi^2r} \int\limits_0^\infty K_0(\xi_1)\,\diff \xi_1 = \frac{1}{2\pi^2r}\cdot \frac{\pi}{2} = \frac{1}{4\pi r}\,.
		\end{aligned}
	\end{equation}
	
	In order to finish the proof, we need to show that
	\begin{equation}\label{eq:x1est}
		\abs{ E_1(x_1,r)} \leq \frac{3}{4\pi\abs{x_1}}\,,
	\end{equation}
	holds since it follows from \eqref{eq:rest} and \eqref{eq:x1est} that
	\begin{equation*}
		\abs{E_1(x_1,r)} \leq \frac{1}{4\pi\max\{\frac{1}{3}\abs{x_1},r\}} \leq \frac{1}{4\pi \frac{1}{\sqrt{2}} \sqrt{\frac{1}{9}{x_1}^2 + r^2}} \leq \frac{3\sqrt{2}}{4\pi\abs{x}}\,.
	\end{equation*}
	
	Suppose that we have proven
	\begin{equation} \label{eq:simplified_estimate}
		\abs{\real  \int\limits_0^\infty \euler^{\ii x\xi}K_0(\xi + \ii y)\,\diff \xi} \leq \frac{3\pi}{2\abs{x}}
	\end{equation}
	for all $x\neq 0$ and $y > 0$. Then, by making the substitution $\xi_1 \mapsto r\xi_1$ we obtain
	\begin{equation*}
		\abs{ \real E_1(x_1,r)}  = \frac{1}{2\pi^2 r}\abs{\real \int\limits_{0}^{\infty}\euler^{\ii (x_1/r)\xi_1}K_0(\xi_1+\ii r)\,\diff \xi_1 } \leq \frac{1}{2\pi^2r} \frac{\pi r}{\abs{2x_1}} = \frac{1}{4\pi\abs{x_1}}\,.
	\end{equation*}
	Thus, the proof can be concluded by verifying \eqref{eq:simplified_estimate}.
	To this end, we observe that we may apply Fubini's theorem to obtain that
	\begin{align*}
		\int\limits_0^\infty \euler^{\ii x\xi}K_0(\xi + \ii y)\,\diff \xi &= \int\limits_0^\infty \int\limits_0^\infty \euler^{\ii x\xi}\euler^{-(\xi+\ii y)\cosh(t)}\,\diff  t\diff \xi \\
		&= \int\limits_0^\infty \bigg( \int\limits_0^\infty \euler^{\xi(\ii x -\cosh(t))}\diff \xi \bigg)\euler^{-\ii y\cosh(t)}\,\diff t \\
		&= \int\limits_0^\infty \frac{\euler^{-\ii y\cosh(t)}\,\diff t}{\cosh(t) - \ii x}\,,
	\end{align*}
	since we have that
	\begin{align*}
		\int\limits_0^\infty \int\limits_0^\infty \abs{\euler^{\ii x\xi}\euler^{-(\xi+\ii y)\cosh(t)}}\,\diff  \xi\diff t = \int\limits_0^\infty \int\limits_0^\infty \euler^{-\xi \cosh(t)}\,\diff\xi\diff t = \int\limits_0^\infty \frac{\diff t}{\cosh(t)} = \frac{\pi}{2} < \infty\,,
	\end{align*}
	We now make the substitution $z = \euler^t$ to obtain
	\begin{equation*}
		\int\limits_0^\infty \frac{\euler^{-\ii y\cosh(t)}\,\diff t}{\cosh(t) - \ii x} = \int\limits_1^\infty \frac{\euler^{-\ii \frac{y}{2}(z+ \frac{1}{z})}}{\frac{1}{2}(z + \frac{1}{z}) - \ii x}\,\frac{\diff z}{z}\,.
	\end{equation*}
	Let us write $\hat \C = \C \cup \{\infty\}$.
	To evaluate the real part of this integral, we consider the meromorphic function
	\begin{equation*}
		f_{x,y}:\C\setminus \{0\} \to \hat \C, \quad z \mapsto x \cdot \frac{\euler^{-\ii \frac{y}{2}(z+ \frac{1}{z})}}{\frac{1}{2}(z + \frac{1}{z}) - \ii x}\cdot\frac{1}{z}\,,
	\end{equation*}
	where we have excluded $0$ since $f_{x,y}$ has an essential singularity there.
	We have that
	\begin{equation*}
		f_{x,y}(z) = 2x \frac{\euler^{-\ii \frac{y}{2}(z+ \frac{1}{z})}}{z^2 - 2\ii xz + 1} =  \frac{2x \cdot \euler^{-\ii \frac{y}{2}(z+ \frac{1}{z})}}{(z - \ii x - i\sqrt{x^2+1})(z- \ii x + i\sqrt{x^2 +1})}\,.
	\end{equation*}
	Let us assume for the moment that $x > 0$. Then, the poles of $f_{x,y}$ lie on the open line segment \[\{z \in \C: \real(z) = 0, \imag(z) > - 1 \}\,. \] 
	
	We wish to show that
	\begin{equation}\label{eq:complex}
		\abs{\real \int\limits_1^\infty f_{x,y}(z)\,\diff z} \leq \frac{3\pi}{2}\,
	\end{equation}
	holds since this easily implies the estimate \eqref{eq:simplified_estimate}. To do so, we observe that for $z \in \R$, we have
	\begin{equation*}
		\overline {f_{x,y}(-z)} = \overline { \frac{\euler^{\ii \frac{y}{2}(z+ \frac{1}{z})}}{ -\frac{1}{2x}(z + \frac{1}{z}) -i}\cdot \frac{1}{-z} } = { \frac{\euler^{-\ii \frac{y}{2}(z+ \frac{1}{z})}}{-\frac{1}{2x}(z + \frac{1}{z}) + i} \cdot \frac{1}{-z} } = f_{x,y}(z)\,.
	\end{equation*}
	Thus, it follows that
	\begin{equation*}
		\real f_{x,y}(z) = \frac{1}{2} \overline{f_{x,y}(z)} + \frac{1}{2} f_{x,y}(z) = \frac{1}{2} f_{x,y}(z) + \frac{1}{2} f_{x,y}(-z)\,.
	\end{equation*}
	Therefore, we can write
	\begin{equation*}
		\real \int\limits_1^\infty f_{x,y}(z)\,\diff z = \real  \lim_{R \to \infty} \frac{1}{2} \bigg( \int\limits_{-R}^{-1}  f_{x,y}(z)\,\diff z + \int\limits_1^R f_{x,y}(z)\,\diff z \bigg) \,.
	\end{equation*}
	To evaluate the integral inside the bracket, we consider the positively oriented semicircles
	\begin{equation*}
		\Gamma_1 = \{ \euler^{\ii \theta}: -\pi \leq \theta \leq 0\}\,, \quad \Gamma_R = \{ R\euler^{\ii \theta}: -\pi \leq \theta \leq 0\}\,.
	\end{equation*}
	If $x > 0$, then we obtain from Cauchy's integral theorem that
	\begin{equation*}\label{eq:cequality}
		\int\limits_{-R}^{-1}  f_{x,y}(z)\,\diff z + \int\limits_1^R f_{x,y}(z)\,\diff z = \int\limits_{\Gamma_R} f_{x,y}(z)\,\diff z -  \int\limits_{\Gamma_1} f_{x,y}(z)\,\diff z\,.
	\end{equation*}
	We deduce the inequality
	\begin{equation}\label{eq:cestimate}
		\biggabs{\real \int\limits_1^\infty f_{x,y}(z)\,\diff z} \leq \frac{1}{2} \lim_{R \to \infty} \abs{\int\limits_{\Gamma_R} f_{x,y}(z)\,\diff z} + \frac{1}{2} \abs{\int\limits_{\Gamma_1} f_{x,y}(z)\,\diff z}\,
	\end{equation}
	and observe that
	\begin{align*}
		\int\limits_{\Gamma_1} f_{x,y}(z)\,\diff z &= \int\limits_{-\pi}^{0} \exp \big[{-\ii \frac{y}{2}(\euler^{\ii \theta} + \euler^{-\ii \theta})}\big]\cdot \bigg[ \frac{1}{2x}(\euler^{\ii \theta} + \euler^{-\ii \theta}) - \ii \bigg]^{-1}\euler^{-\ii \theta}\cdot i\euler^{\ii \theta}\diff \theta \\ &= \int\limits_{-\pi}^{0} \frac{\ii \euler^{-\ii y\cos(\theta)}}{ \frac{1}{x}\cos(\theta) - \ii}\,\diff \theta\,
	\end{align*}
	holds. Since the integrand on the right hand side has a modulus of at most $1$, it follows that
	\begin{equation}\label{eq:ccurve}
		\biggabs{\int\limits_{\Gamma_1} f_{x,y}(z)\,\diff z} \leq \pi\,.
	\end{equation}
	To estimate the integral over $\Gamma_R$, we observe that
	\begin{equation*}
		\sup_{z \in \Gamma_R} \abs{f_{x,y}(z)} \leq \sup_{z \in \Gamma_R} 2x \biggabs{\frac{\euler^{-\ii \frac{y}{2}(z+ \frac{1}{z})}}{z^2 - 2\ii xz + 1}} \leq   \frac{2x}{(R^2+1)}\sup\limits_{\theta \in [-\pi,0]}\euler^{(yR -y/R)\sin(\theta)} \leq \frac{2x}{R^2+1}\,.
	\end{equation*}
	It follows that
	\begin{equation*}
		\biggabs{\int\limits_{\Gamma_R} f_{x,y}(z)\,\diff z} \leq \frac{4 \pi R x}{R^2 + 1} \to 0 \quad (R \to \infty)\,.
	\end{equation*}
	With that, we obtain from \eqref{eq:cestimate} and \eqref{eq:ccurve} that \eqref{eq:complex} holds in the case $x > 0$. In fact, we have shown
	\begin{equation*}
		\biggabs{\real \int\limits_1^\infty f_{x,y}(z)\,\diff z} \leq \frac{\pi}{2}\, \quad (x > 0)\,.
	\end{equation*} In the case $x < 0$ the argument is essentially the same, only that we now have to take the pole at ${z_0 = \ii x - \ii \sqrt{x^2 +1}}$ into consideration. Thus, for $x < 0$, we have  
	\begin{equation*}
		\int\limits_{-R}^{-1}  f_{x,y}(z)\,\diff z + \int\limits_1^R f_{x,y}(z)\,\diff z = \int\limits_{\Gamma_R} f_{x,y}(z)\,\diff z -  \int\limits_{\Gamma_1} f_{x,y}(z)\,\diff z + 2\pi \ii \cdot\mathrm{Res}_{z = z_0} f_{x,y}(z)
	\end{equation*}
	instead of \eqref{eq:cequality}\,.
	
	Taking the limit $R \to \infty$, we obtain
	\begin{equation*}
		\real \int\limits_{1}^{\infty} f_{x,y}(z)\,\diff z = \pi \ii\cdot\mathrm{Res}_{z = z_0} f_{x,y}(z) - \frac{1}{2}\int\limits_{\Gamma_1} f_{x,y}(z)\,\diff z\,.
	\end{equation*}
	To calculate the residue, we remark that
	\begin{align*}
		z_0 + \frac{1}{z_0} &= \ii\bigg(x - \sqrt{x^2 +1} - \frac{1}{x - \sqrt{x^2 +1}} \bigg) \\
		&= \ii\frac{(x - \sqrt{x^2 +1})^2 - 1}{x - \sqrt{x^2 + 1}} \\
		&= \ii\frac{2x^2 - 2x\sqrt{x^2 + 1}}{x - \sqrt{x^2+ 1}} \\
		&= 2\ii x\,.
	\end{align*}
	Thus,
	\begin{equation*}
		\mathrm{Res}_{z = z_0} f_{x,y}(z) = \frac{2x\euler^{-\ii \frac{1}{2}(z_0 + \frac{1}{z_0})}}{z_0 + \ii x - \ii\sqrt{x^2 + 1}} = \ii \frac{x\euler^{yx}}{\sqrt{x^2 +1}}
	\end{equation*}
	and therefore we obtain \eqref{eq:complex} by the following calculation:
	\begin{align*}
		\biggabs{\real \int\limits_1^\infty f_{x,y}(z)\,\diff z} &\leq \frac{1}{2} \lim_{R \to \infty} \biggabs{\int\limits_{\Gamma_R} f_{x,y}(z)\,\diff z} + \frac{1}{2} \biggabs{\int\limits_{\Gamma_1} f_{x,y}(z)\,\diff z} + \pi \abs{\mathrm{Res}_{z = z_0} f_{x,y}(z) } \\ &\leq \frac{\pi}{2} -\pi \frac{x\euler^{yx}}{\sqrt{x^2 +1}} \leq \frac{3}{2} \pi\,. \qedhere
	\end{align*}
\end{proof}
We conclude this appendix by explaining how to obtain Lemma \ref{lem:z_pointwise} from Lemma \ref{3d_pointwise}.
Recall that \[
\mathfrak{p}_z(\xi) = \abs{\xi}^2 + 2z\cdot\xi \quad \xi \in \R^n\,.
\]
We note that
\[
\mathfrak{p}_z(\xi) = (\xi + z)\cdot(\xi +z)\,,
\]
whereas for the symbol $p_\tau$ from the previous section, we have
\[
p_\tau(\xi) = \abs{\xi}^2 + 2\ii \tau \xi_1 - \tau^2 = (\xi + \ii \tau e^{(1)}) \cdot (\xi + \ii \tau e^{(1)})\,,
\]
with $e^{(1)} = (1,0,\dots,0)$. Setting $\tau = \abs{z}/\sqrt{2}$, we can find for each $z \in \C^n$ with $z \cdot z = 0$ an orthogonal matrix $U \in O(n)$ and a vector $v \in \R^n$ of length $\tau$ such that
\begin{equation*}\label{eq:ztau}
	\mathfrak{p}_z(\xi) = p_\tau(U\xi + v) \,.
\end{equation*}
Indeed, since $z \cdot z = 0$ and therefore $\real(z) \perp \imag(z)$ we can  choose $U \in O(n)$ such that ${z = \tau U^{\trans}(\ii e^{(1)} + e^{(2)})}$ and therefore
\begin{equation*}
	\mathfrak{p}_z(\xi) = (U\xi + \tau(\ii e^{(1)} - e^{2})\cdot (U\xi + (\ii e^{(1)} - e ^{(2)}) = p_\tau(U\xi - \tau U e^{(2)})\,.
\end{equation*}
It immediately follows that there exists $v \in \R^n$ of length $\tau$ such that
\begin{equation*}
	\label{eq:Fz_pre} \mathfrak{E}_z = \F^{-1}\mathfrak{p}_z = \e_{v}((\F^{-1}p_\tau) \circ U) = \e_{v}(E_\tau \circ U)\,.
\end{equation*}
Combining this identity with Lemma \ref{3d_pointwise} gives Lemma \ref{lem:z_pointwise}.

\section{Generalization to higher dimensions}\label{sec:extensions}

It is natural to ask in what sense the results of this paper extend to domains in $\R^n$ where $n \geq 4$. In the following remarks, we will give some ideas regarding the challenges that arise in this context.

\begin{enumerate}[a)]
	\item One of the essential ingredients in our proof is the pointwise estimate given in Lemma \ref{3d_pointwise}. If we had a higher-dimensional analog of this estimate, i.e. if there existed a constant $A > 0$ such that for all $\tau > 0$ and $x \in \R^n\setminus\{0\}$,  we had
	\begin{equation}\label{eq:false}
		\abs{E_\tau(x)} \leq \frac{A}{\abs{x}^{n-2}}
	\end{equation}
	it would be straightforward to extend the result of this paper to higher dimensions. Unfortunately, such an estimate does not hold already when $n = 4$. This can be seen as follows: By a calculation similar to the one at the beginning of the proof of Lemma \ref{3d_pointwise}, it follows from the well-known formula
	\begin{equation*}
		K_{1/2}(\lambda) = \bigg(\frac{\pi}{2\lambda}\bigg)^{1/2} \euler^{-\lambda}
	\end{equation*}
	that
	\begin{align*}
		&E_\tau(x) \\ \,&= \frac{1}{(2\pi)^{3/2}} \int\limits_{-\infty}^{\infty} \euler^{\ii x_1 \xi_1} \bigg(\frac{\sqrt{(\xi_1 + \ii \tau)^2}}{r}\bigg)^{1/2} \bigg(\frac{\pi}{2r\sqrt{(\xi_1 + \ii \tau)^2}}\bigg)^{1/2}\euler^{-\ii r \sqrt{(\xi_1 + \ii \tau)^2}}\,\diff \xi_1 \\
		\,&= \frac{1}{2\pi r} \real \int\limits_0^\infty \euler^{\ii x_1 r_1} \euler^{-r(\xi_1 + \ii \tau)}\,\diff \xi_1 \\ \,&= \frac{1}{2\pi r} \real \frac{\euler^{-\ii \tau r }}{r - \ii x_1}  \\
		\,&= \frac{1}{2\pi\abs{x}^2}\bigg(\cos(\tau r) - x_1 \frac{\sin(\tau r)}{r}\bigg)\,.
	\end{align*}
	For $r \sim 0$ it therefore follows that
	\begin{equation*}
		E_\tau(x) \sim \frac{\tau}{2\pi \abs{x_1}}\,.
	\end{equation*}
	If $x_1$ is large, this provides a contradiction to \eqref{eq:false}.
	\item Let $\alpha \in \C$ such that  $0 < \real(\alpha) \leq (n-1)/2$ and conside the tempered distributions $E_\tau^\alpha$ defined by
	\begin{equation}\label{eq:conjecture_definition}
		E_\tau^\alpha = \F_1^{-1}\Phi_\tau^\alpha\,.
	\end{equation}
	By Proposition \ref{prop_fundamentalsolution_fourier}, we have that $E_\tau^\alpha =\Gamma(\alpha)\F^{-1}p_\tau^\alpha$ in the range where $0 < \real(\alpha) < 2$. This suggests using \eqref{eq:conjecture_definition} to obtain an analytic extension of $E_\tau^\alpha$. With suitable pointwise estimates on $E_\tau^\alpha$, one could then try and implement an argument based on Stein interpolation such as the one, for example, in \cite{kenig1987uniform} to prove a generalisation of Lemma \ref{main_lemma}. However, applying such an approach to the family of distributions $E_\tau^\alpha$ requires the resolution of some substantial technical difficulties. This will be the subject of future work.
	
	\item Regarding the question of suitable pointwise estimates for $E_\tau^\alpha$, we conjecture that  an estimate of the form
	\begin{equation}\label{eq:conjecture_estimate}
		\abs{E_\tau^{(n-1)/2 + \ii s}(x)} \leq \frac{C\euler^{C\ii s}}{\abs{x}}
	\end{equation}
	should hold for all $s \in \R$ and $x \in \R^n \setminus \{0\}$. We expect such an estimate to be true, because we have that
	\begin{align*}
		&E_\tau^{(n-1)/2 + \ii s}(x) \\\quad &= \frac{2^{(3-n)/2 - \ii s} s}{(2\pi)^{(n-1)/2}}\int\limits_{-\infty}^{\infty} \euler^{\ii x_1 \xi_1} \bigg( \frac{\sqrt{(\xi_1 + \ii \tau)^2}}{r} \bigg)^{\ii s} K_{\ii s}\bigg(r\sqrt{(\xi_1 + \ii \tau)^2}\bigg)\,\diff \xi_1\,.
	\end{align*}
	By the proof of Lemma \ref{3d_pointwise}, we obtain \eqref{eq:conjecture_estimate} for $s = 0$. Heuristically, the additional terms appearing in the case where $s \neq 0$ should not cause $\abs{E_\tau^{(n-1)/2 + \ii s}(x)}$ to increase substantially. With such an estimate, one would be in a good position to apply a complex interpolation argument.

\end{enumerate}


\small
	\emph{Acknowledgement:} This paper is a revised version of a chapter from the author's PhD thesis \cite{bombach2024}. The author thanks his advisor Peter Stollmann for his guidance and stimulating discussions regarding the subject of this work. Furthermore, the author thanks Martin Tautenhahn, Christian Seifert and Thomas Kalmes for valuable remarks that helped to significantly improve the content of this manuscript.












\newpage 
\small
\bibliographystyle{alpha}
\bibliography{lit}

\end{document}